\documentclass[a4paper,12pt]{article}

\usepackage[left=2cm,right=2cm, top=2cm,bottom=3cm,bindingoffset=0cm]{geometry}

\usepackage{verbatim}
\usepackage{amsmath}
\usepackage{amsthm}
\usepackage{amssymb}
\usepackage{delarray}
\usepackage{cite}
\usepackage{hyperref}
\usepackage{mathrsfs}
\usepackage{tikz}
\usetikzlibrary{patterns}
\usepackage{caption}
\DeclareCaptionLabelSeparator{dot}{. }
\newcommand{\al}{\alpha}

\newcommand{\ga}{\gamma}
\newcommand{\de}{\delta}
\newcommand{\la}{\lambda}

\newcommand{\eps}{\varepsilon}
\newcommand{\vv}{\varphi}

\theoremstyle{plain}

\numberwithin{equation}{section}

\newtheorem{thm}{Theorem}[section]
\newtheorem{lem}[thm]{Lemma}
\newtheorem{prop}[thm]{Proposition}
\newtheorem{cor}[thm]{Corollary}

\theoremstyle{definition}

\newtheorem{example}[thm]{Example}
\newtheorem{ip}[thm]{Inverse Problem}
\newtheorem{ap}[thm]{Auxiliary Problem}
\newtheorem{defin}[thm]{Definition}
\newtheorem{alg}[thm]{Algorithm}
\newtheorem*{quest}{Question}

\theoremstyle{remark}

\usetikzlibrary{arrows.meta}

 \allowdisplaybreaks

\begin{document}

\begin{center}
{\Large\bf Stability of the inverse Sturm-Liouville problem\\[0.2cm] on a quantum tree}
\\[0.5cm]
{\bf Natalia P. Bondarenko}
\end{center}

\vspace{0.5cm}

{\bf Abstract.} This paper deals with the Sturm-Liouville operators with distribution potentials of the space $W_2^{-1}$ on a metric tree. We study an inverse spectral problem that consists in the recovery of the potentials from the characteristic functions related to various boundary conditions. We prove the uniform stability of this inverse problem for potentials in a ball of any fixed radius, as well as the local stability under small perturbations of the spectral data. Our approach is based on a stable algorithm for the unique reconstruction of the potentials relying on the ideas of the method of spectral mappings.

\medskip


\medskip


\section{Introduction} \label{sec:intr}

This paper aims to prove the stability of the inverse Sturm-Liouville problem on a metric tree (i.e. graph without cycles). 
Inverse problems of spectral analysis consist in the reconstruction of differential operators from their spectral characteristics. The basic results of the inverse spectral theory have been obtained for the Sturm-Liouville operators $-y'' + q(x)y$ on finite and infinite intervals (see the monographs \cite{Mar77, Lev84, FY01, Krav20} and references therein).

Differential operators on graphs, also called quantum graphs, are given by the three components: (i) a metric graph, (ii) differential expressions on the graph edges, and (iii) matching conditions at the graph vertices. Spectral problems for such operators arise in numerous applications including organic chemistry, mesoscopic physics and nanotechnology, the theories of photonic crystals, waveguides, quantum chaos (see \cite{Kuch02, PPP04, BCFK06} and references therein). 

Inverse Sturm-Liouville problems that consist in the reconstruction of the potentials on the edges of quantum trees have been investigated in \cite{Bel04, BW05, Yur05, BV06, AK08, ALM10, Bond19} and other studies. In particular, Belishev \cite{Bel04} was the first to prove the uniqueness for solution of an inverse problem on an arbitrary tree. He used the eigenvalues and the derivatives of the eigenfunctions at the boundary vertices as the spectral data. Brown and Weikard \cite{BW05} showed that potentials are uniquely determined by
the Dirichlet-to-Neumann map (i.e. the Weyl matrix), which is equivalent to the spectral data in \cite{Bel04}. 
Yurko \cite{Yur05} has obtained the uniqueness by the minimal amount of data consisting only of the diagonal elements of the Weyl matrix. Moreover, Yurko has developed a constructive approach to solving inverse spectral problems on graphs, based on the method of spectral mappings \cite{Yur02}. Later, this approach was extended to arbitrary compact graphs \cite{Yur10}, as well as to a wide class of noncompact graphs \cite{Ign15}. A more detailed overview of inverse spectral theory for quantum graphs can be found in \cite{Yur16, Kur24}. Among recent advances, we mention the study \cite{But23} of the inverse spectral problem for a nonlocal functional-differential operator on a graph and the numerical reconstruction of potentials on quantum trees in \cite{AKK24}. Thus, although constructive procedures for solving inverse problems on graphs are well-understood, there is a lack of stability analysis. In the recent study \cite{Bond25}, the author has investigated the stability of the inverse Sturm-Liouville problem on a simple graph with a loop. This paper aims to fill the gap for arbitrary tree graphs, whose geometrical structure is definitely more complicated.

In this paper, we consider the Sturm-Liouville problems with distribution potentials of the class $W_2^{-1}$ on a metric tree. 
This class of potentials is wider than the usual $L_1$. In particular, the space $W_2^{-1}$ contains perturbations by the Dirac delta-functions and the Coulomb-type singularities $\frac{1}{x}$, which arise in quantum mechanics (see \cite{Alb05}). Four approaches to defining Sturm-Liouville operators with distribution potentials are described in \cite{SS03}. The Sturm-Lioville operators with $W_2^1$-potentials on metric graphs were considered in \cite{HH20} and inverse problems for such operators, in \cite{FIY08, Vas19, Bond19, Bond21-mmas, Bond21-amp}. Basing on the ideas of those studies, we develop a stable algorithm for the unique reconstruction of the potential from spectral data on a metric tree of arbitrary structure. Moreover, we prove the uniform stability of the inverse problem for potentials in a ball of any fixed radius.

An important role in our investigation is played by analytical characteristic functions, whose zeros coincide with eigenvalues of various Sturm-Liouville problems on the whole tree and on its subtrees. As spectral data, we use characteristic functions related to quantum trees with various boundary conditions. Our problem statement generalizes the two-spectra inverse Sturm-Liouville problem by Borg \cite{Borg46} and the inverse problem by Yurko \cite{Yur05} on a quantum tree with regular potentials. Our stability analysis relies on the Lipschitz continuity of the characteristic functions on the tree with respect to the potentials in suitable norms. This auxiliary result is obtained by using transformation operators that were constructed in \cite{HM04, Bond21, KB25} for the solutions of the Sturm-Liouville equation with $W_2^{-1}$-potential.

Our reconstruction algorithm is based on the following two auxiliary steps: (i) recovery of the potential on a boundary edge, (ii) transition through a vertex, which allows us to pass to internal edges. For step (i), using contour integration, we derive a new reconstruction formula to find distribution potentials on the tree edges and apply it to obtain both uniform and local stability estimates. At step (ii), in contrast to previous studies, we choose a countable set of points in the spectral plane and use sampling to obtain a reconstruction procedure that is valid under any sufficiently small perturbation of the spectral data in a certain space. As a result, we obtain the uniform stability, as well as local stability of the inverse Sturm-Liouville problem on an arbitrary metric tree.

The paper is organized as follows. In Section~\ref{sec:prelim}, we state eigenvalue problems on a metric tree and introduce the notation. In Section~\ref{sec:char}, we describe the construction and some properties of characteristic functions for the Sturm-Liouville problems on the tree. Section~\ref{sec:main} contains the formulations of the main results and briefly outlines their proofs. In Section~\ref{sec:bound}, the uniform stability of recovering the potential on a boundary edge is proved. In Section~\ref{sec:return}, we investigate an auxiliary problem that is related to the transition through a vertex. In Section~\ref{sec:alg}, we develop a stable algorithm for the unique reconstruction of the potentials on the whole tree and prove the corresponding local theorem. Appendix contains the proof of the Lipschitz continuity of the transformation operator kernel with respect to the distributional potential. That auxiliary result has independent significance.  

\section{Preliminaries} \label{sec:prelim}

Denote by $G$ the metric tree that consists of vertices $\{ v_j \}_1^{m+1}$ and edges $\{ e_j \}_1^m$ ($m \ge 1$). Assume that each edge $e_j$ ($j = \overline{1,m}$) goes from the vertex $v_j$ to its parent $v_{p(j)}$, $j < p(j) \le m+1$. The vertex $v_{m+1}$ is the root, it has no outgoing edges. Furthermore, assume that $v_{m+1}$ has the only incoming edge $e_m = [v_m, v_{m+1}]$ (see the example in Fig.~\ref{fig:tree}).
For each vertex $v$, 
denote by $E_v$ the set of all incoming and outgoing edges incident to $v$. The size of $E_v$ is called the degree of $v$. We call $v$ a boundary vertex if it has degree $1$ and an internal vertex otherwise.
Denote the sets of the boundary vertices and of the internal vertices of the tree $G$ by $\partial G$ and $\mbox{int} \, G$, respectively. Since the root $v_{m+1}$ belongs to $\partial G$, we introduce the set $\partial G' := \partial G \setminus v_{m+1}$.

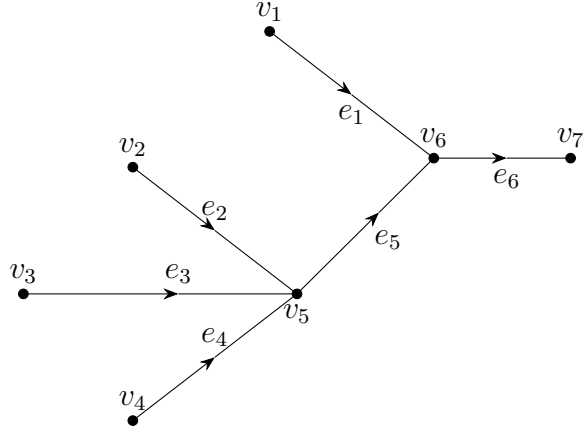
\begin{figure}[h!]
\centering
\begin{tikzpicture}[scale = 1.2]
\filldraw (0, 1) circle (1.5pt) node[anchor=south]{$v_3$};
\filldraw (1.2, -0.4) circle (1.5pt) node[anchor=south]{$v_4$};
\filldraw (1.2, 2.4) circle (1.5pt) node[anchor=south]{$v_2$};
\filldraw (3, 1) circle (1.5pt) node[anchor=north]{$v_5$};
\draw[-{Stealth[length=2mm]}] (0, 1) -- (1.7,1) node[anchor=south]{$e_3$};
\draw (1.7, 1) -- (3, 1);
\draw[-{Stealth[length=2mm]}] (1.2, -0.4) -- (2.1, 0.3) node[anchor=south]{$e_4$};
\draw (2.1, 0.3) -- (3, 1);
\draw[-{Stealth[length=2mm]}] (1.2, 2.4) -- (2.1, 1.7) node[anchor=south]{$e_2$};
\draw (2.1, 1.7) -- (3, 1);
\filldraw (4.5, 2.5) circle (1.5pt) node[anchor=south]{$v_6$};
\draw[-{Stealth[length=2mm]}] (3, 1) -- (3.9, 1.9);
\draw (4, 1.6) node{$e_5$};
\draw (3.9, 1.9) -- (4.5, 2.5);
\filldraw (2.7, 3.9) circle (1.5pt) node[anchor=south]{$v_1$};
\draw[-{Stealth[length=2mm]}] (2.7, 3.9) -- (3.6, 3.2) node[anchor=north]{$e_1$};
\draw (3.6, 3.2) -- (4.5,2.5); 
\filldraw (6, 2.5) circle (1.5pt) node[anchor=south]{$v_7$};
\draw[-{Stealth[length=2mm]}] (4.5, 2.5) -- (5.3, 2.5) node[anchor=north]{$e_6$};
\filldraw (5.3, 2.5) -- (6, 2.5);
\end{tikzpicture}
\caption{Example of tree $G$}
\label{fig:tree}
\end{figure}

Let each edge $e_j$ have length $T_j > 0$ and be parameterized by $x_j \in [0,T_j]$ so that $x_j = 0$ corresponds to the vertex $v_j$ and $x_j = T_j$, to $v_{p(j)}$.
A function on the tree $G$ is a vector $y = \{ y_j \}_1^m$ of complex-valued functions $y_j = y_j(x_j)$, $x_j \in (0,T_j)$, $j = \overline{1,m}$. 

Consider the Sturm-Liouville equations on the tree $G$:
\begin{equation} \label{eqv}
-(y_j^{[1]})' - \sigma_j(x_j) y_j^{[1]} - \sigma_j^2(x_j) y_j = \lambda y_j, \quad x_j \in (0,T_j), \quad j = \overline{1,m},
\end{equation}
where $\sigma = \{ \sigma_j \}_1^m \in L_2(G)$, $y_j^{[1]} := y_j' - \sigma_j y_j$ are the quasi-derivatives, and $\lambda$ is the spectral parameter. We call $\sigma$ the potential on the graph $G$ and $\sigma_j$ the potentials on the corresponding edges $e_j$ ($j = \overline{1,m}$).
Equations \eqref{eqv} are considered for functions $y_j$ such that $y_j, y_j^{[1]} \in AC[0,T_j]$, $j = \overline{1,m}$. It is worth mentioning that equations \eqref{eqv} are equivalent to the Schr\"odinger equations
\begin{equation} \label{schr}
-y_j''(x_j) + q_j(x_j) y_j(x_j) = \la y_j(x_j), \quad x_j \in (0,T_j), 
\end{equation}
with distributional potentials $q_j \in W_2^{-1}(0,T_j)$ (see \cite{SS03}).

For a function $y_j$ on the edge $e_j = [v_j, v_{p(j)}]$, introduce the notations
$$
\begin{cases}
y_j(v_j) = y_j(0), \quad & y_j(v_{p(j)}) = y_j(T_j), \\
y_j^{[1]}(v_j) = -y_j^{[1]}(0), \quad & y_j^{[1]}(v_{p(j)}) = y_j^{[1]}(T_j).
\end{cases}
$$
Thus, $y_j(v)$ is the value of a function on the tree in the vertex $v$ and $y_j^{[1]}(v)$ is the ``normal quasi-derivative'', which is taken in the outgoing direction from the interval $(0, T_j)$. For boundary vertices $v \in \partial G$, we omit the index $j$ and write $y(v)$, $y^{[1]}(v)$.

This paper focuses on the boundary value problems $\mathcal L$ and $\mathcal L_k$ ($v_k \in \partial G'$) for the Sturm-Liouville system \eqref{eqv} with the matching conditions
\begin{equation} \label{mc}
\begin{cases}
y_j(v) = y_k(v), \quad e_j, e_k \in E_v & \text{(continuity condition)}, \\
\sum\limits_{e_j \in E_v} y_j^{[1]}(v) = 0 & \text{(Kirchhoff's condition)}
\end{cases}
\end{equation}
at all the internal vertices $v \in \mbox{int}\,G$ and the corresponding boundary conditions:
\begin{align} \label{bc}
\mathcal L \colon & \quad y(v_j) = 0, \quad v_j \in \partial G, \\ \label{bck}
\mathcal L_k \colon & \quad y^{[1]}(v_k) = 0, \quad y(v_j) = 0, \quad v_j \in \partial G \setminus v_k.
\end{align}

Throughout the paper, we use the following notation:

\begin{itemize}
\item $b = |\partial G|$ is the number of the boundary vertices in the tree $G$.
\item $L_2(G) := \bigl\{ y = \{ y_j \}_1^m \colon y_j \in L_2(0,T_j), \, j = \overline{1,m} \bigr\}$ is the Hilbert space with the scalar product
$$
(y,z) = \sum_{j = 1}^m \int_0^{T_j} \overline{y_j(x_j)} z_j(x_j) \, dx_j
$$
and the norm
$$
\| y \|_{L_2(G)} = \sqrt{\sum_{j = 1}^m \| y_j \|^2_{L_2(0,T_j)}}.
$$
\item For $\Omega > 0$, introduce the ball
$$
\mathcal B_{\Omega} := \bigl\{ y = \{ y_j \}_1^m \in L_2(G) \colon \| y \|_{L_2(G)} \le \Omega \bigr\}.
$$
\item Along with $\sigma = \{ \sigma_j \}_1^m$, we consider other potentials in $L_2(G)$: $\tilde \sigma = \{ \tilde \sigma_j \}_1^m$, $\sigma^0 = \{ \sigma_j^0 \}_1^m$, $\sigma_j^0 \equiv 0$ ($j = \overline{1,m}$), and $\sigma^s = \{ \sigma_j^s \}_{j = 1}^m$ ($s \ge 1$). We agree that, if a symbol $\al$ denotes an object related to the potential $\sigma$, then the notations $\tilde \al$, $\al^0$, and $\al^s$ mean the analogous objects related to the potentials $\tilde \sigma$, $\sigma^0$, and $\sigma^s$, respectively. Also, $\hat \al := \al - \tilde \al$.
\item In estimates, the same symbol $C$ denotes various positive constants independent of $\la$, $x$, etc. In proofs, we assume that constants $C$ depend on the same parameters as in the formulations of the corresponding theorems and lemmas.
\item $\la = \rho^2$, $\mu = \theta^2$.
\item $\mathbf{T} := \mbox{length}(G) = \sum_{j = 1}^m T_j$ is the total length of the graph $G$.
\item For $T > 0$, we denote by $PW(T)$ the class of the Paley-Wiener functions of the form
\begin{equation} \label{PW}
\mathcal F(\rho) = \int_{-T}^T f(t) \exp(i \rho t) \, dt, \quad f \in L_2(-T, T).
\end{equation}
\item The notation $f(\rho) \asymp g(\rho)$ means that
$$
\forall \rho \quad c |f(\rho)| \le |g(\rho)| \le C |f(\rho)|,
$$
where $c$ and $C$ are positive constants.
\end{itemize}

\section{Characteristic functions} \label{sec:char}

In this section, we provide a recursive definition of characteristic functions, whose zeros coincide with the eigenvalues of the corresponding Sturm-Liouville operators on trees. 
Moreover, 
we present the asymptotics of the characteristic functions, and study their continuous dependence on the potential $\sigma$ in suitable norms.

Let $G$ be a metric tree defined in Section~\ref{sec:main}, and let $\sigma = \{ \sigma_j \}_1^m$ be a function on the tree $G$ of the class $L_2(G)$. 
For $j \in \{ 1, 2, \dots, m \}$, denote by $\vv_j(x_j, \la)$ and $S_j(x_j, \la)$ the solutions of equation \eqref{eqv} for the fixed $j$ satisfying the initial conditions:
\begin{equation} \label{ic}
\vv_j(0,\la) = S_j^{[1]}(0,\la) = 1, \quad \vv_j^{[1]}(0,\la) = S_j(0,\la) = 0. 
\end{equation}
Note that equation \eqref{eqv} is equivalent to the first-order system
$$
\frac{d}{dx_j}
\begin{bmatrix}
y_j \\ y_j^{[1]}
\end{bmatrix} 
= \begin{bmatrix}
\sigma_j & 1 \\
-(\sigma_j^2 + \la) & -\sigma_j
\end{bmatrix}
\begin{bmatrix}
y_j \\ y_j^{[1]}
\end{bmatrix}. 
$$
Therefore, the solutions $\vv_j(x_j, \la)$ and $S_j(x_j, \la)$ exist and are unique. Moreover, for each fixed $x_j \in [0,T_j]$, the functions $\vv_j(x_j, \la)$, $\vv^{[1]}_j(x_j, \la)$, $S_j(x_j, \la)$, and $S_j^{[1]}(x_j, \la)$ are entire analytic in $\la$.  

Suppose that the set of the boundary vertices is divided into two disjoint subsets: 
$$
\partial G = \partial G^D \cup \partial G^N, \quad \partial G^D \cap \partial G^N = \varnothing.
$$

Consider the following boundary conditions BC, which are more general than \eqref{bc} and \eqref{bck}:
\begin{equation*}
y(v) = 0, \: v \in \partial G^D, \qquad y^{[1]}(v) = 0, \: v \in \partial G^N.
\end{equation*}
Here and below in similar notations, the upper indices $D$ and $N$ mean ``Dirichlet'' and ``Neumann'', respectively. The quasi-derivatives on the respective edges are defined by using $\sigma$.

\begin{defin}[\hspace*{-3pt}\cite{Bond19}] \label{def:char}
The characteristic function $\Delta(\la) = \Delta(\la; G, \sigma, BC)$ of the Sturm-Liouville problem \eqref{eqv}--\eqref{mc} on the tree $G$ with the boundary conditions BC is defined recursively:
\begin{enumerate}
\item If $m = 1$, then $\Delta(\la)$ is given by the following formulas for various types of boundary conditions:
\begin{equation*}
\def\arraystretch{1.7}
\begin{array}{ll}
y_1(0) = y_1(T_1) = 0 \colon \quad & \Delta(\la) = S_1(T_1, \la), \\
y_1(0) = y_1^{[1]}(T_1) = 0 \colon \quad & \Delta(\la) = S_1^{[1]}(T_1,\la), \\
y_1^{[1]}(0) = y_1(T_1) = 0 \colon \quad & \Delta(\la) = \vv_1(T_1, \la), \\
y_1^{[1]}(0) = y_1^{[1]}(T_1) = 0 \colon \quad & \Delta(\la) = \vv_1^{[1]}(T_1,\la).
\end{array}
\end{equation*}
\item Let $m > 1$, and let $u$ be an internal vertex of degree $r$. Split the tree $G$ by the vertex $u$ into $r$ subtrees $G_j$ ($j = \overline{1,r}$) as in Fig.~\ref{fig:split}. For $j \in \{1, 2, \dots, r\}$, let $\Delta_j^D(\la)$ and $\Delta_j^N(\la)$ be the characteristic functions for the Sturm-Liouville equations \eqref{eqv} on the subtree $G_j$ with the boundary conditions $y(u) = 0$ and $y^{[1]}(u) = 0$, respectively, the conditions BC at the vertices $v \in \partial G_j \setminus u$ and the matching conditions \eqref{mc} at $v \in \mbox{int}\, G_j$. Then
\begin{equation} \label{split}
\Delta(\la) := \sum_{j = 1}^r \Delta_j^N(\la) \prod_{\substack{k = 1 \\ k \ne j}}^r \Delta_k^D(\la).
\end{equation}
This definition does not depend on the choice of $u$.
\end{enumerate}
\end{defin}

\begin{figure}[h!]
\centering
\begin{tikzpicture}
\filldraw (0, 0) circle (2pt) node[anchor=south]{$u$};
\draw (-1, -1) node{$G$};
\filldraw (1, 1) circle (1pt);
\draw (0, 0) edge (1, 1);
\filldraw (1.5, -0.3) circle (1pt);
\draw (0, 0) edge (1.5, -0.3);
\filldraw (2.5, -0.8) circle (1pt);
\draw (1.5, -0.3) edge (2.5, -0.8);
\filldraw (3, 0) circle (1pt);
\draw (1.5, -0.3) edge (3, 0);
\filldraw (3.3, -0.5) circle (1pt);
\draw (2.5, -0.8) edge (3.3, -0.5);
\filldraw (3.6, -0.8) circle (1pt);
\draw (2.5, -0.8) edge (3.6, -0.8);
\filldraw (3.3, -1.1) circle (1pt);
\draw (2.5, -0.8) edge (3.3, -1.1);
\filldraw (0.5, -1) circle (1pt);
\draw (0, 0) edge (0.5, -1);
\filldraw (1.3, -2) circle (1pt);
\draw (0.5, -1) edge (1.3, -2);
\filldraw (0.7, -2) circle (1pt);
\draw (0.5, -1) edge (0.7, -2);
\filldraw (-1.5, 0.5) circle (1pt);
\draw (0, 0) edge (-1.5, 0.5);
\filldraw (-2.5, 1) circle (1pt);
\draw (-1.5, 0.5) edge (-2.5, 1);
\filldraw (-2.8, 0.5) circle (1pt);
\draw (-1.5, 0.5) edge (-2.8, 0.5);
\filldraw (-2.5, 0) circle (1pt);
\draw (-1.5, 0.5) edge (-2.5, 0);
\filldraw (8, 0.3) circle (2pt) node[anchor=south]{$u$};
\filldraw (9, 1.3) circle (1pt);
\draw (8, 0.3) edge (9, 1.3);
\draw (8.6, 1.1) node{$G_1$};
\filldraw (8, 0) circle (2pt) node[anchor=west]{$u$};
\filldraw (9.5, -0.3) circle (1pt);
\draw (8, 0) edge (9.5, -0.3);
\draw (9.4, 0) node{$G_2$}; 
\filldraw (10.5, -0.8) circle (1pt);
\draw (9.5, -0.3) edge (10.5, -0.8);
\filldraw (11, 0) circle (1pt);
\draw (9.5, -0.3) edge (11, 0);
\filldraw (11.3, -0.5) circle (1pt);
\draw (10.5, -0.8) edge (11.3, -0.5);
\filldraw (11.6, -0.8) circle (1pt);
\draw (10.5, -0.8) edge (11.6, -0.8);
\filldraw (11.3, -1.1) circle (1pt);
\draw (10.5, -0.8) edge (11.3, -1.1);
\filldraw (8, -0.3) circle (2pt) node[anchor=north]{$u$};
\filldraw (8.5, -1.3) circle (1pt);
\draw (8, -0.3) edge (8.5, -1.3);
\filldraw (9.3, -2.3) circle (1pt);
\draw (8.5, -1.3) edge (9.3, -2.3);
\filldraw (8.7, -2.3) circle (1pt);
\draw (8.5, -1.3) edge (8.7, -2.3);
\draw (8.2, -1.3) node{$G_3$};
\filldraw (7.7, 0) circle (2pt) node[anchor=south]{$u$};
\filldraw (6.2, 0.5) circle (1pt);
\draw (7.7, 0) edge (6.2, 0.5);
\filldraw (5.2, 1) circle (1pt);
\draw (6.2, 0.5) edge (5.2, 1);
\filldraw (4.9, 0.5) circle (1pt);
\draw (6.2, 0.5) edge (4.9, 0.5);
\filldraw (5.2, 0) circle (1pt);
\draw (6.2, 0.5) edge (5.2, 0);
\draw (5.2, -0.5) node{$G_4$};
\end{tikzpicture}
\caption{Splitting of the tree $G$ by the vertex $u$}
\label{fig:split}
\end{figure}
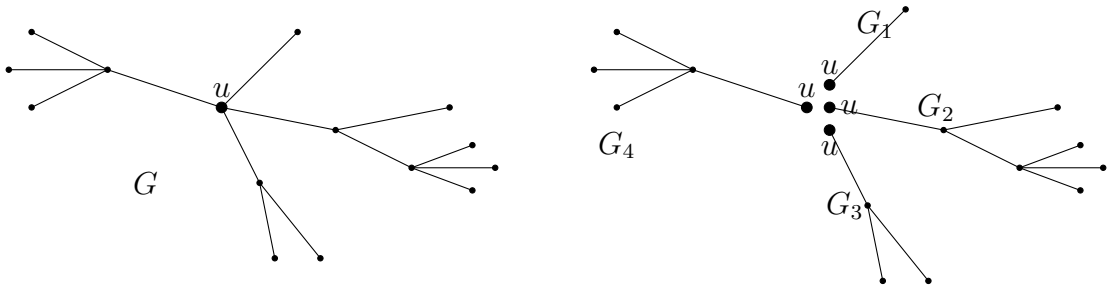

\begin{prop} \label{prop:char}
The spectrum of the Sturm-Liouville problem \eqref{eqv}--\eqref{mc} on the tree $G$ with the boundary conditions BC is a countable set of eigenvalues, which coincide (counting with multiplicities) with the zeros of the characteristic function $\Delta(\la)$ constructed in Definition~\ref{def:char}. 
\end{prop}

In the case of the Sturm-Liouville equations with regular potentials, our construction of the characteristic function inductively follows from \cite[Theorem 2.1]{LP09}, which implies Proposition~\ref{prop:char}. For equations \eqref{eqv}, the construction and the proof are analogous. In a number of previous studies (see, e.g., \cite{Yur05, Yur10, Vas19}), characteristic functions were constructed as determinants of some linear systems. That definition is equivalent to ours up to a constant multiplier. Anyway, we find the recursive definition of characteristic functions more convenient for investigating their properties.

Proceed to the asymptotics of characteristic functions.

\begin{prop}[\hspace*{-3pt}\cite{HM04}; \cite{Bond21}, Theorem~2.1] \label{prop:SC}
For $j = \overline{1,m}$, the following relations hold:
\begin{align} \label{intS}
& S_j(x,\rho^2) = \frac{\sin \rho x}{\rho} + \frac{\psi_j(x,\rho)}{\rho}, \quad S_j^{[1]}(x,\rho^2) = \cos \rho x + \xi_j(x, \rho), \\ \label{intvv}
& \vv_j(x,\rho^2) = \cos \rho x + \zeta_j(x, \rho), \quad \vv_j^{[1]}(x,\rho^2) = -\rho \sin \rho x + \rho \eta_j(x,\rho) + c_j(x),
\end{align}
where the functions $\psi_j(x,\rho)$, $\xi_j(x,\rho)$, $\zeta_j(x,\rho)$, and $\eta_j(x,\rho)$ are continuous by $x \in [0, T_j]$ and, for each fixed $x \in (0,T_j]$,
belong to $PW(x)$, and $c_j(x) = \vv_j^{[1]}(x,0)$.
\end{prop}

\begin{lem} \label{lem:Delta0}
The characteristic function $\Delta^0(\la) := \Delta(\la; G, 0, BC)$ for the case of zero potentials $\sigma_j^0 \equiv 0$ $(j = \overline{1,m})$ has the form
\begin{equation} \label{D0}
\Delta^0(\rho^2) = \rho^{1-d} P_{m.\mathbf{T}}(\rho),
\end{equation}
where $d = |\partial G^D|$ is the number of the Dirichlet boundary conditions among BC, and $P_{m,\mathbf{T}}(\rho)$ is a trigonometric polynomial composed of the functions $\sin \rho T_j$ and $\cos \rho T_j$ $(j = \overline{1,m})$ with the following properties:

(i) All the zeros of $P_{m,\mathbf{T}}(\rho)$ are real. 

(ii) $P_{m,\mathbf{T}}(\rho) \asymp \exp(|\mbox{Im} \rho| \mathbf{T})$ in the sectors $\epsilon \le \pm \arg \rho \le \pi - \epsilon$ for any fixed $\epsilon > 0$.
\end{lem}

\begin{proof}
For $\sigma_j^0 \equiv 0$, equation \eqref{eqv} turns into $-y_j'' = \lambda y_j$ and the quasi-derivative $y_j^{[1]}$ coincides with the standard derivative $y_j'$. Therefore, we have $S_j^0(x_j,\rho^2) = \dfrac{\sin \rho x_j}{\rho}$ and $\vv_j^0(x_j, \rho^2) = \cos \rho x_j$. Using Definition~\ref{def:char}, we get the representation \eqref{D0} by induction. According to Proposition~\ref{prop:char}, the zeros of $\Delta^0(\la)$ coincide with the eigenvalues of the Laplacian $-\dfrac{d^2}{dx^2}$ on the tree $G$. It can be easily shown that this Laplacian is self-adjoint and non-negative definite, which implies the property (i) of $P_{m,\mathbf{T}}$. The property (ii) is proved by obtaining the asymptotics of the characteristic function inductively in the corresponding sectors. 
\end{proof}

Using Definition~\ref{def:char}, Proposition~\ref{prop:SC}, and Lemma~\ref{lem:Delta0}, we get the following corollary.

\begin{cor} \label{cor:asymptD}
For $d = |\partial G^D| > 0$,
the characteristic function $\Delta(\la) = \Delta(\la; G, \sigma, BC)$ for $\sigma \in L_2(G)$ admits the representation
\begin{equation} \label{asymptD}
\Delta(\rho^2) = \Delta^0(\rho^2) + \rho^{1-d} \kappa(\rho), \quad \kappa \in PW(\mathbf{T}),
\end{equation}
where $\Delta^0(\la) = \Delta(\la; G, 0, BC)$. If $d = 0$, then the constant $\Delta(0)$ has to be added to the right-hand side of \eqref{asymptD}.
\end{cor}

\begin{example} \label{ex:star}
The characteristic function for the Sturm-Liouville problem on the star-shaped graph in Fig.~\ref{fig:star} with the Dirichlet boundary conditions \eqref{bc} has the form
$$
\Delta(\la) = \sum_{j = 1}^{m-1} S_j^{[1]}(T_j, \la) \prod_{\substack{k = 1 \\ k \ne j}}^m S_k(T_k,\la) + C_m(T_m,\la) \prod_{k = 1}^{m-1} S_k(T_k,\la)
$$
and the asymptotics
$$
\Delta(\rho^2) = \rho^{1-m} \left( \sum_{j = 1}^m \cos \rho T_j \prod_{\substack{k = 1 \\ k \ne j}}^m \sin \rho T_k + \kappa(\rho) \right).
$$
\end{example}

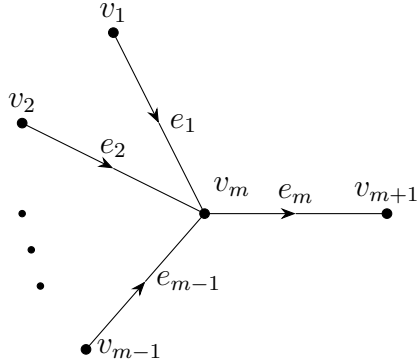
\begin{figure}[h!]
\centering
\begin{tikzpicture}[scale = 1.2]
\filldraw (1, 2) circle (1.5pt) node[anchor=south]{$v_1$};
\filldraw (0, 1) circle (1.5pt) node[anchor=south]{$v_2$};
\filldraw (0.7, -1.5) circle (1.5pt) node[anchor=west]{$v_{m-1}$};
\filldraw (2, 0) circle (1.5pt);
\draw (2.3, 0.3) node{$v_m$};
\filldraw (4, 0) circle (1.5pt) node[anchor=south]{$v_{m+1}$};
\draw[-{Stealth[length=2mm]}] (1, 2) -- (1.5, 1) node[anchor=west]{$e_1$};
\draw (1.5, 1) -- (2, 0);
\draw[-{Stealth[length=2mm]}] (0, 1) -- (1, 0.5) node[anchor=south]{$e_2$};
\draw (1, 0.5) -- (2, 0);
\draw[-{Stealth[length=2mm]}] (0.7, -1.5) -- (1.35, -0.75) node[anchor=west]{$e_{m-1}$};
\draw (1.35, -0.75) -- (2, 0);
\draw[-{Stealth[length=2mm]}] (2, 0) -- (3, 0) node[anchor=south]{$e_m$};
\draw (3, 0) -- (4, 0);
\filldraw (0, 0) circle (1pt);
\filldraw (0.2, -0.8) circle (1pt);
\filldraw (0.1, -0.4) circle (1pt);
\end{tikzpicture}    
\caption{Star-shaped graph in Example~\ref{ex:star}}
\label{fig:star}
\end{figure}

Using transformation operators for the solutions $S_j(x_j, \la)$ and $\vv_j(x_j, \la)$ (see \cite{HM04, Bond21}), we obtain the following lemma on the Lipschitz continuity of these solutions and their quasi-derivatives with respect to the potential $\sigma_j$.

\begin{lem} \label{lem:uniSC}
Let $\Omega > 0$ and $j \in \{ 1, 2, \dots, m \}$ be fixed. Then, for any complex-valued functions $\sigma_j$ and $\tilde \sigma_j$ in $L_2(0,T_j)$ such that
$$
\| \sigma_j \|_{L_2(0,T_j)} \le \Omega, \quad \| \tilde \sigma_j \|_{L_2(0,T_j)} \le \Omega,
$$
and for each fixed $x \in (0,T_j]$, there hold
\begin{align} \label{uniS}
& \| \rho \hat S_j(x,\rho^2) \|_{L_2(\mathbb R)} \le C \| \hat \sigma_j \|_{L_2(0,x)}, \quad \| \hat S_j^{[1]}(x, \rho^2) \|_{L_2(\mathbb R)} \le C \| \hat \sigma_j \|_{L_2(0,x)}, \\ \label{univv}
& \| \hat \vv_j(x, \rho^2) \|_{L_2(\mathbb R)} \le C \| \hat \sigma_j \|_{L_2(0,x)}, \quad \| \hat \eta_j(x,\rho) \|_{L_2(\mathbb R)} + |\hat c_j(x)| \le C \| \hat \sigma_j \|_{L_2(0,x)},
\end{align}
where $\eta_j(x,\rho)$ and $c_j(x)$ are the functions from \eqref{intvv} and the constant $C$ depends only on $\Omega$ and $j$.
\end{lem}

The proof of Lemma~\ref{lem:uniSC} is presented in Appendix. Lemma~\ref{lem:uniSC} and Definition~\ref{def:char} imply the following corollary.

\begin{cor} \label{cor:uniD}
Let $\Omega > 0$ be fixed.
Suppose that $\sigma$ and $\tilde \sigma$ lie in $\mathcal B_{\Omega}$, $\Delta(\la) = \Delta(\la; G, \sigma, BC)$ and $\tilde \Delta(\la) = \Delta(\la; G, \tilde\sigma, BC)$. If $d > 0$, then 
$$
\| \rho^{d-1} \hat \Delta(\rho^2) \|_{L_2(\mathbb R)} \le C \| \hat\sigma \|_{L_2(G)}
$$
where the constant $C$ depends only on $\Omega$. 
In other words, the remainder term $\kappa(\rho)$ in the asymptotics \eqref{asymptD} is Lipschitz continuous by the potential $\sigma$ as the mapping from $L_2(G)$ to $L_2(\mathbb R)$.
In the case $d = 0$, we similarly get
$$
\| \hat \kappa \|_{L_2(\mathbb R)} + |\hat \Delta(0)| \le C \| \sigma \|_{L_2(G)}.
$$
\end{cor}

According to part (i) of Lemma~\ref{lem:Delta0}, the zeros of $\Delta^0(\la)$ are real and non-negative. Consequently, taking Corollaries~\ref{cor:asymptD} and~\ref{cor:uniD} into account, we get the following properties for the zeros of $\Delta(\la)$.

\begin{cor} \label{cor:reg}
The zeros of the characteristic function $\Delta(\la) = \Delta(\la; G, \sigma, BC)$ lie in the region
\begin{equation} \label{reg}
\mbox{Re} \, \la \le -C, \quad |\mbox{Im} \, \sqrt{\la}| \le C. 
\end{equation}
For all $\sigma \in \mathcal B_{\Omega}$, the same constant $C$ in \eqref{reg} can be chosen depending only on $\Omega$.
\end{cor}

\section{Main results and proof strategy} \label{sec:main}


Denote by $\Delta(\la)$ and $\Delta_k(\la)$ ($v_k \in \partial G'$) the characteristic functions of the boundary value problems problems $\mathcal L$ and $\mathcal L_k$, respectively, constructed according to Definition~\ref{def:char}. We study the following inverse spectral problem.

\begin{ip} \label{ip:main}
Given $\Delta(\la)$ and $\Delta_k(\la)$ for all $v_k \in \partial G'$, find $\{ \sigma_j \}_1^m$.
\end{ip}

Inverse Problem~\ref{ip:main} is equivalent to the recovery of the potentials $\{ \sigma_j \}_1^m$ from the spectra of the eigenvalue problems $\mathcal L$ and $\mathcal L_k$ ($k = \overline{1,m}$), and so generalizes the Borg problem \cite{Borg46}. The solution of Inverse Problem~\ref{ip:main} is unique, which is proved similarly to \cite{Yur05, FIY08, Vas19}. 

The first result of this paper is the following theorem, which establishes the uniform stability of Inverse Problem~\ref{ip:main} for potentials $\sigma$ in the ball $\mathcal B_{\Omega}$.

\begin{thm} \label{thm:uni}
Suppose that $\Omega > 0$. Then, for any $\sigma = \{ \sigma_j \}_1^m$ and $\tilde \sigma = \{ \tilde \sigma_j \}_1^m$ in $\mathcal B_{\Omega}$, there holds
\begin{equation} \label{estsi}
\| \hat \sigma_j \|_{L_2(0,T_j)} \le C \de, \quad j = \overline{1,m},
\end{equation}
where 
\begin{equation} \label{defde}
\de := \| \rho^{b-1} \hat \Delta(\rho^2)\|_{L_2(\mathbb R)} +
\sum_{v_k \in \partial G'} \| \rho^{b-2} \hat \Delta_k(\rho^2) \|_{L_2(\mathbb R)}
\end{equation}
and the constant $C$ depends only on $\Omega$.
\end{thm}

Let us outline the proof of Theorem~\ref{thm:uni}.
Fix an index $k$ of a boundary vertex $v_k \in \partial G'$ and consider the following auxiliary inverse problem on the edge $e_k$:

\begin{ip} \label{ip:loc}
Given $\Delta(\la)$ and $\Delta_k(\la)$, find $\sigma_k$.
\end{ip}

First, we prove the uniform stability of Inverse Problem~\ref{ip:loc}:

\begin{thm} \label{thm:unibound}
Let $\Omega > 0$ and an index $k$ such that $v_k \in \partial G'$ be fixed. Then, for any $\sigma$ and $\tilde \sigma$ in $\mathcal B_{\Omega}$, there holds
\begin{equation} \label{unik}
\| \hat \sigma_k \|_{L_2(0,T_k)} \le C \de_k,
\end{equation}
where
\begin{equation} \label{defdek}
\de_k := \| \rho^{b-1} \hat \Delta(\rho^2) \|_{L_2(\mathbb R)} + \| \rho^{b-2} \hat \Delta_k(\rho^2) \|_{L_2(\mathbb R)}.
\end{equation}
and $C$ depends only on $\Omega$ and $k$.
\end{thm}

In order to prove Theorem~\ref{thm:unibound}, we derive a new reconstruction formula in Section~\ref{sec:bound} within the framework of the method of spectral mappings (see \cite{Yur02}). 

Next, let $v_p$ be the parent of $v_k$: $p = p(k)$, $e_k = [v_k, v_p]$. Denote by $g_p$ the directed subtree of $G$ with the root $v_p$ and by $G_p$ the directed subtree obtained from $G$ by excluding $g_p$ (see Fig.~\ref{fig:tree-split}). Thus $v_p$ is a boundary vertex in $G_p$. Denote by $\Delta_p^D(\la)$ and $\Delta_p^N(\la)$ the characteristic functions constructed by Definition~\ref{def:char} for the Sturm-Liouville problems \eqref{eqv} on the subtree $G_p$ with the boundary conditions $y_p(0) = 0$ and $y^{[1]}_p(0) = 0$, respectively, the Dirichlet conditions $y(v) = 0$ at the other boundary vertices $v \in \partial G_p \setminus v_p$, and the matching conditions \eqref{mc} at the internal vertices $v \in \mbox{int} \, G_p$. Consider the following auxiliary problem.

\begin{figure}
\centering
\begin{tikzpicture}
\filldraw (0, 1) circle (1.5pt) node[anchor=south]{$v_k$};
\filldraw (-0.5, 3) circle (1.5pt);
\filldraw (-0.5, 2) circle (1.5pt);
\draw[-{Stealth[length=2mm]}] (-0.5, 3) -- (0.35, 2.7);
\draw (0.35, 2.7) edge (1.2, 2.4);
\draw[-{Stealth[length=2mm]}] (-0.5, 2) -- (0.35, 2.2);
\draw (0.35, 2.2) edge (1.2, 2.4);
\filldraw (1.2, -0.4) circle (1.5pt);
\filldraw (1.2, 2.4) circle (1.5pt);
\filldraw (3, 1) circle (1.5pt) node[anchor=north]{$v_p$};
\draw[-{Stealth[length=2mm]}] (0, 1) -- (1.7,1);
\draw (1.7, 1) -- (3, 1);
\draw[-{Stealth[length=2mm]}] (1.2, -0.4) -- (2.1, 0.3);
\draw (2.1, 0.3) -- (3, 1);
\draw[-{Stealth[length=2mm]}] (1.2, 2.4) -- (2.1, 1.7);
\draw (2.1, 1.7) -- (3, 1);
\filldraw (4.5, 2.5) circle (1.5pt);
\draw[-{Stealth[length=2mm]}] (3, 1) -- (3.9, 1.9);
\draw (4, 1.6);
\draw (3.9, 1.9) -- (4.5, 2.5);
\filldraw (2.7, 3.9) circle (1.5pt);
\draw[-{Stealth[length=2mm]}] (2.7, 3.9) -- (3.6, 3.2);
\draw (3.6, 3.2) -- (4.5,2.5); 	
\filldraw (6, 2.5) circle (1.5pt) node[anchor=south]{$v_{m+1}$};
\draw[-{Stealth[length=2mm]}] (4.5, 2.5) -- (5.3, 2.5);
\filldraw (5.3, 2.5) -- (6, 2.5);
\draw[dashed] (4.4, 2.7) circle (2.14cm);
\draw[dashed] (0.6, 1) circle (2.4cm);
\draw (5,4) node{$G_p$};
\draw (0,-0.5) node{$g_p$};
\end{tikzpicture}
\caption{Graphs $G_p$ and $g_p$}
\label{fig:tree-split}
\end{figure}
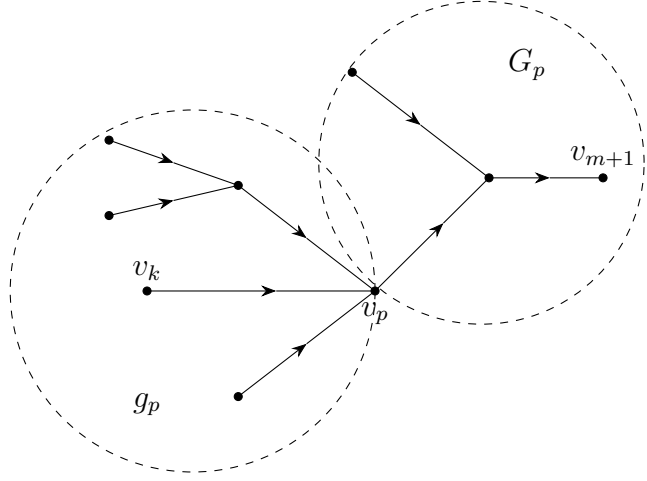

\begin{ap} \label{ap:aux}
Given the potentials $\sigma_j$ on all the edges $e_j$ of $g_p$ and the characteristic functions $\Delta(\la)$ and $\Delta_k(\la)$, find $\Delta_p^D(\la)$ and $\Delta_p^N(\la)$.
\end{ap}

Denote by $B_p$ the number of the boundary vertices in $G_p$.
Then, the uniform stability of Auxiliary Problem~\ref{ap:aux} holds in the following sense:

\begin{thm} \label{thm:uniaux}
Suppose that $\Omega > 0$. Then, for any $\sigma = \{ \sigma_j \}_1^m$ and $\tilde \sigma = \{ \tilde \sigma_j \}_1^m$ in $\mathcal B_{\Omega}$, there holds
\begin{equation} \label{uniaux}
\| \rho^{B_p-1} \hat \Delta_p^D(\rho^2) \|_{L_2(\mathbb R)} + 
\| \rho^{B_p-2} \hat \Delta_p^N(\rho^2) \|_{L_2(\mathbb R)} \le C \bigl( \| \hat \sigma \|_{L_2(g_p)} + \de_k \bigr),
\end{equation}
where $\de_k$ is defined in \eqref{defdek} and the constant $C$ depends only on $\Omega$ and $k$.
\end{thm}

Theorems~\ref{thm:unibound} and~\ref{thm:uniaux} are proved in Sections~\ref{sec:bound} and~\ref{sec:return}, respectively. Note that, using the resulting functions $\Delta_p^D(\la)$ and $\Delta_p^N(\la)$ of Auxiliary Problem~\ref{ap:aux}, one can construct the potential $\sigma_p$ by solving Inverse Problem~\ref{ip:loc} for the subtree $G_p$. In this way, starting from the boundary edges, one can step-by-step recover the potentials on all the edges of the tree $G$. Consequently, Theorems~\ref{thm:unibound} and \ref{thm:uniaux} directly imply Theorem~\ref{thm:uni}.

In addition, we consider the question about local solvability and stability of Inverse Problem~\ref{ip:main}:

\begin{quest}
Let $\tilde \Delta(\la)$ and $\tilde \Delta_k(\la)$ for all $v_k \in \partial G'$ be obtained from the characteristic functions $\Delta(\la)$ and $\Delta_k(\la)$ corresponding to some potential $\sigma \in L_2(G)$ by a small perturbation in the sense of the norm \eqref{defde}. Does the solution $\{ \tilde \sigma_j \}_1^m$ of Inverse Problem~\ref{ip:main} for the initial data $\tilde \Delta(\la)$ and $\tilde \Delta_k(\la)$ $(v_k \in \partial G')$ exist? If yes, is it stable in the sense of the estimates \eqref{estsi}?
\end{quest}

The answer to the existence question in general is negative, since Inverse Problem~\ref{ip:main} is overdetermined. Nevertheless, in this paper, we develop a theoretical algorithm for reconstruction (Algorithm~\ref{alg:main}), which for every fixed $\sigma \in L_2(G)$ and any sufficiently small perturbation of the spectral data of $\sigma$ generates the unique approximation $\tilde \sigma \in L_2(G)$ satisfying \eqref{estsi}. Consequently, we establish the local stability of Inverse Problem~\ref{ip:main} under the existence assumption for its solution:

\begin{thm} \label{thm:locex}
Let $\sigma = \{ \sigma_j \}_1^m$ be a fixed potential in $L_2(G)$. Then, there exists $\eps > 0$ (depending on $\sigma$) such that, for any $\tilde \sigma = \{ \tilde \sigma_j \}_1^m \in L_2(G)$ satisfying the condition $\de < \eps$, where $\de$ is defined in \eqref{defde}, the estimate \eqref{estsi} holds with the constant $C$ depending only on $\sigma$.
\end{thm}

Note that, for each fixed $\sigma \in L_2(G)$ and every $\eps > 0$, there exist infinitely many functions $\tilde \sigma$ satisfying the condition $\de < \eps$ of Theorem~\ref{thm:locex}. This follows from the continuity of the characteristic functions $\Delta(\la)$ and $\Delta_k(\la)$ with respect to $\sigma$ (see Corollary~\ref{cor:uniD}). Taking a small perturbation of $\sigma$, one can make $\de$ arbitrarily small.

\section{Stability on a boundary edge} \label{sec:bound}

In this section, we consider Inverse Problem~\ref{ip:loc} of the recovery $\sigma_k$ on a boundary edge $e_k$.
Introduce the Weyl function with respect to the vertex $v_k$ according to \cite{Yur05}:
\begin{equation} \label{defM}
M_k(\la) := -\frac{\Delta_k(\la)}{\Delta(\la)}.
\end{equation}

The uniqueness of recovering the potential $\sigma_k$ from the Weyl function $M_k(\la)$ is obtained similarly to \cite[Theorem~2]{FIY08}.
The goal of this section is to prove Theorem~\ref{thm:unibound} on the uniform stability. For this purpose, we derive the new reconstruction formula \eqref{recsi}. The novelty of our case consists in the following two features:
\begin{enumerate}
\item Behavior of the spectrum for the Sturm-Liouville operator on a graph is complicated. Eigenvalues can be multiple and/or not asymptotically separated. For simple structure of the spectrum, it is convenient to calculate the contour integrals using the Residue Theorem and then operate with infinite series (see, e.g., \cite{Bond21}). In our case, we only work with contour integrals.
\item The potentials $q_j$ in the Schr\"odinger-form equation \eqref{schr} are distributional, which influence the convergence of the contour integrals. In order to obtain the desired reconstruction formula, we approximate $q_j$ by integrable potentials.
\end{enumerate}

Consider two problems $\mathcal L$ and $\tilde{\mathcal L}$ on the same three $G$ with potentials $\sigma = \{ \sigma_j \}_1^m$ and $\tilde{\sigma} = \{ \tilde \sigma_j \}_1^m$ of $L_2(G)$, respectively. 

For a fixed $\tau > 0$, introduce the contours
\begin{equation} \label{cont}
\ga := \{ \rho = s + i \tau \colon -\infty < s < \infty \}, \quad
\Gamma := \{ \la = \rho^2 \colon \rho \in \gamma \}
\end{equation}
with the circuit corresponding to $s$ going from $+\infty$ to $-\infty$.
In view of Corollary~\ref{cor:reg},
one can choose $\tau > 0$ so large that all the zeros of $\Delta(\la)$ and $\tilde \Delta(\la)$ lie inside $\Gamma$ (see Fig.~\ref{fig:Gamma}).

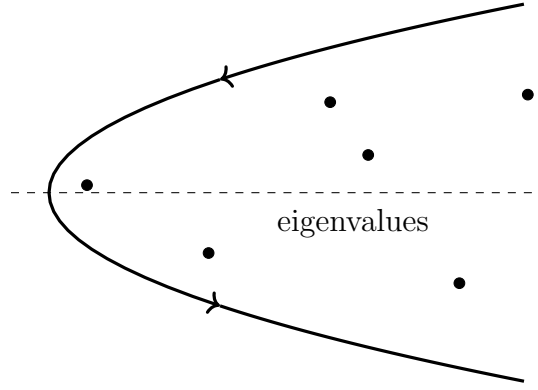
\begin{figure}[h!]
\centering
\begin{tikzpicture}
  \draw[thin,dashed] (-0.5, 0) edge (6.5, 0);
  \draw[very thick,domain=-2.5:-1.5] plot (\x*\x,\x);
  \draw[very thick,domain=-1.5:1.5,<-] plot (\x*\x,\x);
  \draw[very thick,domain=1.5:2.5,<-] plot (\x*\x,\x);
  \filldraw (0.5, 0.1) circle(2pt);
  \filldraw (2.1, -0.8) circle(2pt);
  \filldraw (3.7, 1.2) circle(2pt);
  \filldraw (4.2, 0.5) circle (2pt);
  \filldraw (5.4, -1.2) circle (2pt);
  \filldraw (6.3, 1.3) circle (2pt);
  \draw (4,-0.4) node{eigenvalues};
\end{tikzpicture}
\caption{Contour $\Gamma$}
\label{fig:Gamma}
\end{figure}

Let us derive the relation for the difference $\hat \sigma_k(x)$. Begin from the case 
\begin{equation} \label{siW21}
\sigma_j \in \mathring{W}_2^1[0,T_j], \quad j = \overline{1,m},
\end{equation}
where 
$$
\mathring{W}_2^1[0,T] = \bigl\{ \sigma \in AC[0,T] \colon \sigma' \in L_2(0,T), \, \sigma(0) = \sigma(T) = 0 \bigr\}.
$$
Then, equations \eqref{eqv} can be represented in the form \eqref{schr} with
\begin{equation} \label{qL2}
q_j = \sigma_j' \in L_2(0,T_j), \quad \int_0^{T_j} q_j(x_j) \, dx_j = 0, 
\end{equation}
and $y_j^{[1]}(x_j) = y_j'(x_j)$ at $x_j = 0$ and $x_j = T_j$, $j = \overline{1,m}$. Thus, the boundary value problem $\mathcal L$ turns into the Sturm-Liouville problem with regular potentials and the standard matching conditions, which has been studied in \cite{Yur05}.
In particular, by using the method of spectral mappings (see \cite{Yur02, Yur05}), the following relation has been obtained:
$$
\hat q_k(x) = \frac{1}{2 \pi i} \int_{\Gamma} \frac{d}{dx}\bigl(S_k(x,\la) \tilde S_k(x,\la)\bigr) \hat M_k(\la) \, d\la.
$$

Integration implies
\begin{equation} \label{difsi1}
\hat \sigma_k(x) = \frac{1}{2 \pi i} \int_{\Gamma} S_k(x,\la) \tilde S_k(x,\la) \hat M_k(\la) \, d\la,
\end{equation}
where the constant of integration is chosen according to the condition $\sigma_k(0) = \tilde \sigma_k(0) = 0$. Note that, for the potentials $\{ q_j \}_1^m$ and $\{ \tilde q_j \}_1^m$ satisfying the conditions \eqref{qL2}, there hold
$$
|S_k(x,\rho^2)|, \, |\tilde S_k(x,\rho^2)| \le C |\rho|^{-1}, \quad \rho \in \ga
$$
and $\hat M_k(\rho^2) \in L_2(\ga)$ (see \cite[Section~4]{Bond25}). Consequently, the integral
\begin{equation} \label{intga}
\int_{\Gamma} S_k(x,\la) \tilde S_k(x,\la) \hat M_k(\la) \, d\la = 2 \int_{\ga} \rho S_k(x, \rho^2) \tilde S_k(x,\rho^2) \hat M_k(\rho^2) d \rho
\end{equation}
converges absolutely and uniformly by $x \in [0,T_k]$. 

\begin{lem} \label{lem:int0}
For $\{ q_j \}_1^m$ and $\{ \tilde q_j \}_1^m$ satisfying \eqref{qL2}, there holds:
\begin{equation} \label{intMk}
\int_{\Gamma} \frac{\hat M_k(\la)}{\la} \, d\la = 0.
\end{equation}
\end{lem}

\begin{proof}
Introduce the contours (see Fig.~\ref{fig:cont}):
\begin{equation} \label{defcont}
\mathcal C_N = \{ \la \in \mathbb C \colon |\la| = R_N \}, \quad c_N = \mathcal C_N \cap \mbox{int} \, \Gamma, \quad \Gamma_N := \Gamma \cap \mbox{int} \, \mathcal C_N,
\end{equation}
where the radii $\{ R_N \}$ are chosen so that $R_N \to +\infty$ as $N \to \infty$ and
\begin{equation} \label{Dbelow}
|\Delta(\la)|, \, |\tilde \Delta(\la)| \ge c |\rho|^{1-b} \exp(|\mbox{Im} \, \rho| \mathbf{T}), \quad c > 0, \quad b = |\partial G|, \quad \la = \rho^2 \in \mathcal C_N.
\end{equation}

\begin{figure}[h!]
\centering
\begin{tikzpicture}[scale = 0.7]
  \draw[very thick,domain=-2.5:-1.5] plot (\x*\x,\x);
  \draw[very thick,domain=-1.5:1.5,<-] plot (\x*\x,\x);
  \draw[very thick,domain=1.5:2.5,<-] plot (\x*\x,\x);
  \filldraw (0.5, 0.1) circle(2pt);
  \filldraw (2.1, -0.8) circle(2pt);
  \filldraw (3.7, 1.2) circle(2pt);
  \filldraw (4.2, 0.5) circle (2pt);
  \filldraw (5.4, -1.2) circle (2pt);
  \filldraw (6.3, 1.3) circle (2pt);
  \draw[very thick,->] (-2.5, 0) arc (-180:180:3.5cm); 
  \draw[very thick,->] (-2.5, 0) arc (180:360:3.5cm); 
  \draw (-0.2, -0.8) node{$\Gamma_N$};
  \draw (5, 0.2) node{$c_N$};
  \draw (-1.8, 0.3) node{$\mathcal C_N$}; 
\end{tikzpicture}
\caption{Contours $\mathcal C_N$, $c_N$, and $\Gamma_N$}
\label{fig:cont}
\end{figure}
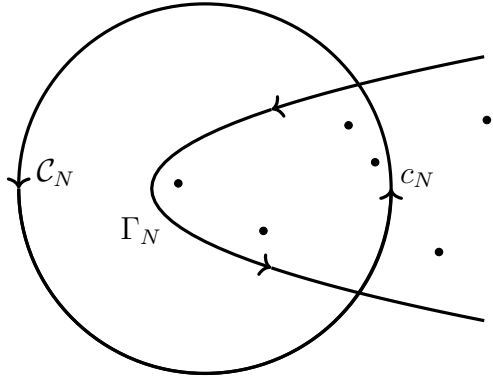

Taking the condition \eqref{qL2} into account, similarly to Corollary~\ref{cor:asymptD}, we obtain the asymptotics
\begin{equation} \label{asymptDreg}
\begin{cases}
\Delta(\rho^2) = \Delta^0(\rho^2) + o\bigl(\rho^{-b} \exp(|\mbox{Im} \, \rho| \mathbf{T})\bigr), \\
\Delta_k(\rho^2) = \Delta_k^0(\rho^2) + o\bigl(\rho^{1-b} \exp(|\mbox{Im} \, \rho| \mathbf{T})\bigr), 
\end{cases} \quad |\rho| \to \infty.
\end{equation}
Combining \eqref{defM}, \eqref{Dbelow}, and \eqref{asymptDreg} implies
$$
\lim_{N \to \infty} \sup_{\la \in \mathcal C_N} |\hat M_k(\la)| = 0.
$$
Hence
\begin{equation} \label{intCN}
\lim_{N \to \infty} \int_{c_N} \frac{\hat M_k(\la)}{\la} \, d\la = 0, \quad \lim_{N \to \infty} \oint_{\mathcal C_N} \frac{\hat M_k(\la)}{\la} \, d\la = 0.
\end{equation}
Consequently
$$
\int_{\Gamma} \frac{\hat M_k(\la)}{\la} \, d\la = \lim_{N \to \infty} \int_{\Gamma_N} \frac{\hat M_k(\la)}{\la} \, d\la = \lim_{N \to \infty} \oint_{\Gamma_N \cup c_N} \frac{\hat M_k(\la)}{\la} \, d\la.
$$
In view of the Residue Theorem, we have
$$
\oint_{\Gamma_N \cup c_N} \frac{\hat M_k(\la)}{\la} \, d\la = \oint_{\mathcal C_N} \frac{\hat M_k(\la)}{\la} \, d\la.
$$
This together with  \eqref{intCN} conclude the proof.
\end{proof}

Combining \eqref{difsi1} and \eqref{intMk}, we arrive at the relation
\begin{equation} \label{recsi}
\hat \sigma_k(x) = \frac{1}{2\pi i} \int_{\Gamma} \left( S_k(x, \la) \tilde S_k(x, \la) - \frac{1}{2 \la} \right) \hat M_k(\la) \, d\la
\end{equation}
with the absolutely convergent integral in the case of $\sigma$ and $\tilde \sigma$ satisfying \eqref{siW21}.

\begin{lem} \label{lem:rec}
The relation \eqref{recsi} is valid for any $\sigma$ and $\tilde \sigma$ in $L_2(G)$ with an appropriately chosen contour $\Gamma$.
\end{lem}

\begin{proof}
First, let us study the convergence of the integral in \eqref{recsi}. In view of Corollary~\ref{cor:asymptD}, we get
\begin{gather} \label{asymptDk}
\Delta(\rho^2) = \Delta^0(\rho^2) + \rho^{1-b} \kappa(\rho), \quad
\Delta_k(\rho^2) = \Delta_k^0(\rho^2) + \rho^{2-b} \kappa_k(\rho), \quad \kappa, \kappa_k \in PW(\mathbf{T}), \\ \label{estDk}
\Delta(\rho^2) \asymp \rho^{1-b}, \quad \Delta_k(\rho^2) \asymp \rho^{2-b}, \quad \rho \in \gamma,
\end{gather}
and the similar relations for $\tilde \Delta(\rho^2)$ and $\tilde \Delta_k(\rho^2)$. 
In view of \eqref{PW}, any function of $PW(\mathbf{T})$ belong to $L_2(\ga)$
and
\begin{equation} \label{eqPW}
\| \mathcal F \|_{L_2(\mathbb R)} \asymp \| \mathcal F \|_{L_2(\ga)}, \quad \forall \mathcal F \in PW(\mathbf{T}).
\end{equation}
Therefore, using \eqref{defM}, we conclude that 
\begin{equation} \label{L2ga}
\rho^{-1} \hat M_k(\rho^2) \in L_2(\ga).
\end{equation}

Substituting the relations \eqref{intS} into \eqref{recsi} and taking \eqref{intga} into account, we obtain
\begin{gather} \label{I12}
\frac{1}{2\pi i} \int_{\Gamma} \left( S_k(x, \la) \tilde S_k(x, \la) - \frac{1}{2 \la} \right) \hat M_k(\la) \, d\la = \mathcal I_1(x) + \mathcal I_2(x), \\ \label{defI12}
\mathcal I_1(x) := 
-\frac{1}{2\pi i} \int_{\ga} \cos(2 \rho x) \rho^{-1} \hat M_k(\rho^2) \, d\rho, \quad
 \mathcal I_2(x) := \frac{1}{\pi i}\int_{\ga} \varkappa(x,\rho) \rho^{-1} \hat M_k(\rho^2) \, d\rho, \\ \label{defvarka}
 \varkappa(x,\rho) = ( \psi_k(x, \rho) + \tilde \psi_k(x, \rho)) \sin \rho x + \psi_k(x,\rho) \tilde \psi_k(x,\rho),
\end{gather}
where $\psi_k(x,\rho)$ is the function from the representation \eqref{intS}. Consequently $\varkappa(x,\rho) \in L_2(\ga)$ for each fixed $x \in [0,T_k]$ and the norm $\| \varkappa(x,.) \|_{L_2(\ga)}$ is uniformly bounded by $x$. In view of \eqref{L2ga}, the integral $\mathcal I_1(x)$ is the Fourier transform of an $L_2(\ga)$-function, so $\mathcal I_1 \in L_2(0,T_k)$. The integral $\mathcal I_2$ converges absolutely and uniformly by $x$ to a continuous function on $[0,T_k]$. Thus, the integral in the right-hand side of \eqref{recsi} is correctly defined.

Second, let us prove the relation \eqref{recsi} for the general case by approximation. For any functions $\sigma = \{ \sigma_j \}_1^m$ and $\tilde \sigma = \{ \tilde \sigma_j \}_1^m$ of $L_2(G)$, there exist sequences $\sigma^s = \{ \sigma_j^s \}_1^m$ and $\tilde \sigma^s = \{ \tilde \sigma_j^s \}_1^m$ ($s \ge 1$) of the class \eqref{siW21} such that $\sigma^s \to \sigma$ and $\tilde \sigma^s \to \tilde \sigma$ in $L_2(G)$ as $s \to \infty$. By Corollary~\ref{cor:uniD}, we conclude that $\rho^{b-1} \Delta^s(\rho^2) \to \rho^{b-1} \Delta(\rho^2)$ and $\rho^{b-2} \Delta_k^s(\rho^2) \to \rho^{b-2} \Delta_k(\rho^2)$ in $L_2(\mathbb R)$ (and so in $L_2(\ga)$) as $s \to \infty$. Using \eqref{defM}, we obtain $\rho^{-1}M_k^s(\rho^2) \to \rho^{-1} M_k(\rho^2)$ in $L_2(\ga)$ as $s \to \infty$. Lemma~\ref{lem:uniSC} implies that $\rho S_k^s(x,\rho^2) \to \rho S_k(x,\rho^2)$ in $L_2(\ga)$ as $s \to \infty$ uniformly by $x \in [0,T_k]$. 

Under the conditions \eqref{siW21}, we have already proved the relations
\begin{equation} \label{recsip}
\hat \sigma_k^s = \frac{1}{2\pi i} \int_{\Gamma} \left( S_k^s(x,\la) \tilde S_k^s(x,\la) - \frac{1}{2\la} \right) \hat M_k^s(\la) \, d\la.
\end{equation}
Note that one can choose the same contour $\Gamma$ for all $s \ge 1$, since the sequences $\{ \sigma^s \}_1^{\infty}$ and $\{ \tilde\sigma^s \}_1^{\infty}$ are convergent and so bounded.

Passing to the limit as $s \to \infty$, we get
$$
\frac{1}{2\pi i} \int_{\Gamma} \left( S_k^s(x,\la) \tilde S_k^s(x,\la) - \frac{1}{2\la} \right) \hat M_k^s(\la) \, d\la \to \frac{1}{2\pi i} \int_{\Gamma} \left( S_k(x,\la) \tilde S_k(x,\la) - \frac{1}{2\la} \right) \hat M_k(\la) \, d\la.
$$
More precisely, according to \eqref{I12}, the Fourier transforms $\mathcal I_1^s$ converge to $\mathcal I_1$ in $L_2(0,T_k)$ and the continuous functions $\mathcal I_2^s$ converge to $\mathcal I_2$ uniformly by $x \in [0,T_k]$. Thus \eqref{recsip} implies \eqref{recsi}.
\end{proof}

\begin{proof}[Proof of Theorem~\ref{thm:unibound}]
By virtue of Corollary~\ref{cor:uniD}, the remainders $\kappa(\rho)$ and $\kappa_k(\rho)$ in the asymptotics \eqref{asymptDk} are continuous with respect to $\sigma$ as mappings from $L_2(G)$ to $L_2(\mathbb R)$ and so are bounded for $\sigma \in \mathcal B_{\Omega}$.
Consequently, the estimates \eqref{estDk} hold uniformly by $\sigma \in \mathcal B_{\Omega}$ on the appropriate contour $\ga$. So, using \eqref{defM}, \eqref{estDk}, and \eqref{eqPW}, we obtain
\begin{equation} \label{estMk}
\| \rho^{-1} \hat M_k(\rho^2) \|_{L_2(\ga)} \le C \de_k.
\end{equation}

The estimate \eqref{estMk} immediately yields $\| \mathcal I_1 \|_{L_2(0,T_k)} \le C \de_k$, where $\mathcal I_1$ is defined in \eqref{defI12}. Next, by virtue of Lemma~\ref{lem:uniSC}, the norm $\| \psi_k(x,.) \|_{L_2(\mathbb R)}$ is uniformly bounded for $\| \sigma_k \|_{L_2(0,T_k)} \le \Omega$ and $x \in [0,T_k]$. Consequently, the norm $\| \varkappa(x,.) \|_{L_2(\ga)}$ of the function defined by \eqref{defvarka} is also uniformly bounded. Therefore, using \eqref{defI12}, \eqref{estMk}, and the Cauchy-Bunyakovsky-Schwarz inequality, we conclude that $\| \mathcal I_2 \|_{L_2(0,T_k)} \le C \de_k$. Taking \eqref{recsi} and \eqref{I12} into account, we arrive at the estimate \eqref{unik}.
\end{proof}

\section{Uniform stability of the auxiliary problem} \label{sec:return}

In this section, we study the stability of Auxiliary Problem~\ref{ap:aux}. Our analysis is based on constructing characteristic functions for the subtrees $G_p$ and $g_p$ (see Fig.~\ref{fig:tree-split}) and applying their Lipschitz continuity with respect to the potentials.

Let $v_k \in \partial G$ be fixed, and let $v_p$ be the parent of $v_k$.
For the Sturm-Liouville equations \eqref{eqv} on the subtree $g_p$, introduce the characteristic functions $\Delta^{DD}(\la)$, $\Delta^{DK}(\la)$, $\Delta^{ND}(\la)$, and $\Delta^{NK}(\la)$ according to Definition~\ref{def:char} associated with the following boundary conditions:
\begin{itemize}
\item $\Delta^{DD}(\la)$: the Dirichlet conditions $y_k(v_k) = 0$ and $y_j(v_p) = 0$ for all the ingoing edges $e_j$ to $v_p$;
\item $\Delta^{DK}(\la)$: the Dirichlet condition $y_k(v_k) = 0$ and the matching conditions \eqref{mc} at $v_p$;
\item $\Delta^{ND}(\la)$: the Neumann condition $y_k^{[1]}(v_k) = 0$ and the Dirichlet conditions $y_j(v_p) = 0$ for all the ingoing edges $e_j$ to $v_p$;
\item $\Delta^{NK}(\la)$: the Neumann condition $y_k^{[1]}(v_k) = 0$ and the matching conditions \eqref{mc} at $v_p$;
\end{itemize}
wherein $D$, $N$, and $K$ mean ``Dirichlet'', ``Neumann'', and ``Kirchhoff'', respectively. 
It is assumed that, in the internal vertices $v \in \mbox{int}\,g_p \setminus v_p$, the matching conditions \eqref{mc} and, in the boundary vertices $v \in \partial g_p \setminus \{ v_k, v_p \}$, the Dirichlet conditions $y(v) = 0$ are imposed. Note that, in the case of the Dirichlet conditions at $v_p$, the tree $g_p$ actually splits into several subtrees with the root at $v_p$. 

Using \eqref{split}, we obtain the relations
\begin{equation} \label{sys}
\begin{cases}
\Delta(\la) = \Delta^{DD}(\la) \Delta_p^N(\la) + \Delta^{DK}(\la) \Delta_p^D(\la), \\
\Delta_k(\la) = \Delta^{ND}(\la) \Delta_p^N(\la) + \Delta^{NK}(\la) \Delta_p^D(\la).
\end{cases}
\end{equation}
Cramer's rule implies
\begin{equation} \label{cram}
\Delta_p^N(\la) = \frac{A_1(\la)}{A(\la)}, \quad \Delta_p^D(\la) = \frac{A_2(\la)}{A(\la)},
\end{equation}
where
\begin{align} \label{defE}
& A(\la) := \Delta^{DD}(\la) \Delta^{NK}(\la) - \Delta^{ND}(\la) \Delta^{DK}(\la), \\ \label{defE1}
& A_1(\la) := \Delta(\la) \Delta^{NK}(\la) - \Delta_k(\la) \Delta^{DK}(\la), \\ \label{defE2}
& A_2(\la) := \Delta^{DD}(\la) \Delta_k(\la) - \Delta^{ND}(\la) \Delta(\la).
\end{align}

In the special case of the graph $g_p$ consisting of the only edge $e_k = [v_k, v_p]$, we have
\begin{align*}
& \Delta^{DD}(\la) = S_k(T_k,\la), \quad \Delta^{DK}(\la) = S_k^{[1]}(T_k,\la), \\
& \Delta^{ND}(\la) = \vv_k(T_k,\la), \quad \Delta^{NK}(\la) = \vv_k^{[1]}(T_k,\la), \\
& A(\la) = S_k(T_k,\la) \vv_k^{[1]}(T_k,\la) - S_k^{[1]}(T_k,\la) \vv_k(T_k,\la).
\end{align*}

Using equation \eqref{eqv}, one can easily show that the generalized Wronskian $W_k(x,\la) = (S_k \vv_k^{[1]} - S_k^{[1]} \vv_k)(x,\la)$ does not depend on $x$. The initial conditions \eqref{ic} imply $W_k(0,\la) \equiv -1$, so $A(\la) = W_k(T_k,\la) \equiv -1$. Thus, the case of a single edge in $g_p$ is quite simple. In the following, we focus on the case of $g_p$ containing more than one edge.

Exclude the edge $e_k$ from the graph $g_p$ and denote the resulting graph by $g_p^*$. Let $\Delta^D(\la)$ and $\Delta^K(\la)$ be the characteristic function of the Sturm-Liouville system \eqref{eqv} on the graph $g_p^*$ with the Dirichlet conditions $y_j(v_p) = 0$ for all $e_j = [v_j, v_p]$ and the matching conditions \eqref{mc} at the vertex $v_p$, respectively. We assume that the Dirichlet boundary conditions $y(v) = 0$ are imposed at the other boundary vertices $v \in \partial g_p^* \setminus v_p$ and the matching conditions \eqref{mc} are satisfied at all $v \in \mbox{int} \, g_p^*$.

\begin{lem} \label{lem:E}
There holds $A(\la) = -\bigl( \Delta^D(\la) \bigr)^2$.
\end{lem}

\begin{proof}
Considering the split of the subtree $g_p$ by the vertex $v_p$ and using the formula \eqref{split}, we derive the relations
\begin{align*}
& \Delta^{DD}(\la) = S_k(T_k,\la) \Delta^D(\la), \quad 
\Delta^{DK}(\la) = S_k(T_k,\la) \Delta^K(\la) + S_k^{[1]}(T_k,\la) \Delta^D(\la), \\
& \Delta^{ND}(\la) = \vv_k(T_k,\la) \Delta^D(\la), \quad \Delta^{NK}(\la) = \vv_k(T_k,\la) \Delta^K(\la) + \vv_k^{[1]}(T_k,\la) \Delta^D(\la).
\end{align*}
Substituting them into \eqref{defE} and using $W_k(T_k,\la) \equiv -1$, we get the claim.
\end{proof}

According to Corollary~\ref{cor:asymptD}, we have the asymptotics
\begin{equation} \label{asymptDp}
\Delta_p^N(\rho^2) = \Delta_p^{N,0}(\rho^2) + \rho^{2-B_p} \kappa_p^N(\rho), \quad
\Delta_p^D(\rho^2) = \Delta_p^{D,0}(\rho^2) + \rho^{1-B_p} \kappa_p^D(\rho),
\end{equation}
where $\kappa_p^N(\rho)$ and $\kappa_p^D(\rho)$ are Paley-Wiener functions of the class $PW(\mathcal T)$, $\mathcal T := \mbox{length}(G_p)$.

Now, we are ready to prove the uniform stability of Auxiliary Problem~\ref{ap:aux}.

\begin{proof}[Proof of Theorem~\ref{thm:uniaux}]
Denote by $b_p$ the number of the boundary vertices of the tree $g_p$. Then, the tree $g_p^*$ has $(b_p-1)$ boundary vertices not counting $v_p$. Using Corollaries~\ref{cor:asymptD} and~\ref{cor:uniD}, and taking \eqref{eqPW} into account, we get the following uniform estimates for $\sigma, \, \tilde \sigma \in \mathcal B_{\Omega}$:
\begin{gather} \label{estDelD}
\Delta^D(\rho^2) \asymp \rho^{1-b_p}, \quad \rho \in \ga, \qquad
\| \rho^{b_p-1} \hat \Delta^D(\rho^2) \|_{L_2(\ga)} \le c \|\hat\sigma \|_{L_2(g_p^*)}, \\ \label{estDelNK}
\Delta^{NK}(\rho^2) \asymp \rho^{2-b_p}, \quad \rho \in \ga, \qquad
\| \rho^{b_p-2} \hat \Delta^{NK}(\rho^2) \|_{L_2(\ga)} \le C \| \hat\sigma \|_{L_2(g_p)}, \\ \label{estDelDK}
\Delta^{DK}(\rho^2) \asymp \rho^{1-b_p}, \quad \rho \in \ga, \qquad
\| \rho^{b_p-1} \hat \Delta^{DK}(\rho^2) \|_{L_2(\ga)} \le C \|
\hat\sigma \|_{L_2(g_p)}. 
\end{gather}

Lemma~\ref{lem:E} together with \eqref{estDelD} imply
\begin{equation} \label{estE}
A(\rho^2) \asymp \rho^{2(1-b_p)}, \quad \rho \in \ga, \quad
\| \rho^{2(b_p-1)} \hat A(\rho^2) \|_{L_2(\ga)} \le C \| \hat\sigma \|_{L_2(g_p^*)}.
\end{equation}

Using \eqref{estDk}, \eqref{defE1}, \eqref{estDelNK}, and \eqref{estDelDK}, we obtain
\begin{equation} \label{estE1}
|A_1(\rho^2)| \le C |\rho|^{3-b-b_p}, \quad \rho \in \ga, \quad
\| \rho^{3 - b - b_p} \hat A_1(\rho^2) \|_{L_2(\ga)} \le C \bigl( \| \hat\sigma \|_{L_2(g_p)} + \de_k \bigr),
\end{equation}
where $\de_k$ is defined in \eqref{defdek}.

Note that $B_p-1 + b_p = b$. So, using \eqref{cram}, \eqref{estE}, and \eqref{estE1}, we get
$$
\| \rho^{B_p-1} \hat \Delta_p^D(\rho^2) \|_{L_2(\ga)} \le C \bigl( \| \hat \sigma \|_{L_2(g_p)} + \de_k\bigr).
$$
Analogously, we obtain the estimate for $\| \rho^{B_p-2} \hat \Delta_p^N(\rho^2) \|_{L_2(\ga)}$. Taking \eqref{asymptDp} and \eqref{eqPW} into account, we arrive at \eqref{uniaux}.
\end{proof}

\section{Stable algorithm} \label{sec:alg}

In this section, we develop an algorithm for unique and stable reconstruction of the potential $\sigma$ on the tree $G$.
Moreover, we prove Theorem~\ref{thm:locex} on the local stability of Inverse Problem~\ref{ip:main}.

\subsection{Reconstruction on a boundary edge}

For reconstruction on the boundary edges, we derive the main equation of the method of spectral mappings. Our technique is quite similar to \cite{FIY08}, so we outline it briefly. 

Consider a potential $\sigma \in L_2(G)$ and the zero potential $\sigma^0 \equiv 0$. Let an index $k$ of a boundary vertex $v_k \in \partial G'$ be fixed. Choose the contour $\Gamma$ according to \eqref{cont} so that the zeros of $\Delta(\la)$ and $\Delta^0(\la)$ lie inside $\Gamma$. Furthermore, choose radii $\{ R_N \}$ so that the estimates \eqref{Dbelow} hold for $\Delta(\la)$ and $\Delta^0(\la)$ and define the contours \eqref{defcont}. The contour integration implies the relation
$$
S_k^0(x,\la) = S_k(x, \la) + \frac{1}{2\pi i}\lim_{N \to \infty} \oint_{\Gamma_N \cup c_N} D_k^0(x, \la, \mu) (M_k - M_k^0)(\mu) S_k(x,\mu) \, d\mu, \quad x \in [0,T_k],
$$
where
$$
D_k^0(x,\la,\mu) := \int_0^x S_k^0(x,\la) S_k^0(x,\mu) \, dx.
$$

Recall that $S_k^0(x,\la) = \dfrac{\sin \rho x }{\rho}$, $\la = \rho^2$, $\mu = \theta^2$.
For fixed $x \in [0,\pi]$, $\la$, and $\mu \in c_N$, the following estimates hold:
$$
|S_k(x, \mu)| \le \frac{C}{|\theta|}, \quad
|D_k^0(x, \la, \mu)| \le \frac{C}{|\rho||\theta|}, \quad
|(M_k - M_k^0)(\mu)| \le C|\theta|.
$$
Here and below, constants $C$ depend only on $\sigma$. Also, we have $\mbox{length}(c_N) \sim \sqrt{R_N}$. Consequently
$$
\lim_{N \to \infty} \int_{c_N} D_k^0(x, \la, \mu) (M_k - M_k^0)(\mu) S_k(x,\mu) \, d\mu = 0.
$$
Hence
$$
S_k^0(x,\la) = S_k(x, \la) + \frac{1}{2\pi i} \int_{\Gamma} D_k^0(x, \la, \mu) (M_k - M_k^0)(\mu) S_k(x,\mu) \, d\mu
$$
for $x \in [0,T_k]$ and $\la \in \mathbb C$.
Representing $S_k(x, \la)$ in the form \eqref{intS} and passing to the contour $\ga$ in the $\rho$-plane, we obtain
\begin{equation} \label{maineq}
\psi_k(x, \rho) + \frac{1}{\pi i} \int_{\ga} r_k^0(x, \rho, \theta) \psi_k(x, \theta) \, d\theta = F_k(x, \rho),
\end{equation}
where
\begin{gather} \label{defrk0}
r_k^0(x, \rho, \theta) := \mathscr M_k(\theta) \int_0^x \sin \rho t \sin \theta t \, dt, \quad \mathscr M_k(\theta) := \theta^{-1} (M_k - M_k^0)(\theta^2), \\ \label{defFk}
F_k(x, \rho) := \int_0^x f_k(x, t) \sin \rho t \, dt, \quad
f_k(x, t) := -\frac{1}{\pi i} \int_{\ga} \mathscr M_k(\theta) \sin \theta x \sin \theta t \, d \theta.
\end{gather}

For each fixed $x \in [0,T_k]$, the relation \eqref{maineq} can be considered as an equation in $L_2(\ga)$ with respect to $\psi_k(x,.)$.
Analogously to \eqref{L2ga}, we get $\mathscr M_k \in L_2(\ga)$. Therefore, we obtain
$$
\int_{\ga} \int_{\ga} |r_k^0(x, \rho, \theta)|^2 |d\theta| |d\rho| \le C,
$$
similarly to \cite[Lemma~4]{FIY08}. Furthermore, we have $f_k(x, .) \in L_2(0,x)$ and so $F_k(x, .) \in L_2(\ga)$. 

Thus, equation \eqref{maineq} can be rewritten in the form
\begin{equation} \label{mainop}
(I + H_k^0(x)) \psi_k(x) = F_k(x), \quad x \in [0,T_k],
\end{equation}
where $I$ is the identity operator in $L_2(\ga)$ and
\begin{equation} \label{defHk0}
H_k^0(x) \kappa(\rho) := \frac{1}{\pi i} \int_{\ga} r_k^0(x, \rho, \theta) \kappa(\theta) \, d \theta
\end{equation}
is the Hilbert-Schmidt operator in $L_2(\ga)$ for each fixed $x \in [0,T_k]$. 

Analogously to \cite[Theorem 3]{FIY08}, one can show that $(I + H_k^0(x))$ has a bounded inverse operator on $L_2(\ga)$:
\begin{equation} \label{invHk0}
(I + H_k^0(x))^{-1} = I - H_k(x),
\end{equation}
where
\begin{gather*}
H_k(x) \kappa(\rho) := \frac{1}{\pi i} \int_{\ga} r_k(x, \rho, \theta) \kappa(\theta) \, d \theta, \\
r_k(x, \rho, \theta) := \rho \theta \mathscr M_k(\theta) \int_0^x S_k(t,\rho^2) S_k(t,\theta^2) \, dt.
\end{gather*}
Hence, the main equation \eqref{maineq} is uniquely solvable. Using its solution $\psi_k(x, \rho)$, one can find $S_k(x,\la)$ by \eqref{intS} and then recover the potential by the formula
\begin{equation} \label{recsi0}
\sigma_k(x) = \frac{1}{2 \pi i} \int_{\Gamma} \left( S_k(x, \la) S_k^0(x, \la) - \frac{1}{2\la}\right)(M_k - M_k^0)(\la) \, d\la,
\end{equation}
which is a special case of \eqref{recsi}.

\subsection{Transition through a vertex}

The solution of Auxiliary Problem~\ref{ap:aux} is based on the linear algebraic system \eqref{sys}. However, not every small perturbation of $\Delta(\la)$ and $\Delta_k(\la)$ preserves a solution $\Delta_p^N(\la)$ and $\Delta_p^D(\la)$ of \eqref{sys} in the class of entire analytic functions having the asymptotics \eqref{asymptDp}. Therefore, in order to develop a stable algorithm for the unique reconstruction, we use the sampling approach.

Introduce the points
\begin{equation} \label{defnu}
\nu_n := \frac{\pi n}{\mathcal T} + i \tau, \quad \mu_n = \nu_n^2, \quad n \in \mathbb Z,
\end{equation}
where $\tau > 0$ is sufficiently large. Then, the sequence $e_n(t) := \exp(i \nu_n t)$ ($n \in \mathbb Z$) is a Riesz basis in $L_2(-\mathcal T, \mathcal T)$. We readily obtain the following analog of Whittaker-Kotel'nikov-Shannon Theorem, which is a special case of Kramer's Lemma (see \cite[Theorem~2.1]{ABF09} and \cite{Hig96}).

\begin{lem} \label{lem:Riesz}
The mapping of a function $\mathcal F$ to the sequence $\{ \mathcal F(\nu_n) \}_{n \in \mathbb Z}$ is a linear isomorphism between $PW(\mathcal T)$ (with the $L_2(\mathbb R)$-norm) and $l_2$. The inverse mapping is given by the formula
\begin{equation} \label{expF}
\mathcal F(\rho) = \sum_{n = -\infty}^{\infty} \mathcal F(\nu_n) \frac{\sin (\rho + \overline{\nu}_n) \mathcal T}{(\rho + \overline{\nu}_n) \mathcal T}.
\end{equation}
\end{lem}

\begin{proof}
The biorthonormal basis to $\{ e_n \}_{n \in \mathbb Z}$ consists of the functions $e_n^*(t) := \frac{1}{2\mathcal T}\exp(i \overline{\nu}_n t)$ ($n \in \mathbb Z$):
$$
(e_n, e_k^*) = \int_{-\mathcal T}^{\mathcal T} \overline{e_n(t)} e_k^*(t) \, dt =
\begin{cases}
    1, & n = k, \\
    0, & n \ne k.
\end{cases}.                
$$

In view of \eqref{PW}, the values $\{ \mathcal F(\nu_n) \}_{n \in \mathbb Z}$ are the coordinates of the function $f$ in the basis $\{ e_n^* \}_{n \in \mathbb Z}$:
\begin{equation} \label{expf}
f(t) = \sum_{n = -\infty}^{\infty} \mathcal F(\nu_n) e_n^*(t).
\end{equation}
Substituting \eqref{expf} into \eqref{PW}, we arrive at \eqref{expF}.

Applying Plancherel's theorem
\begin{equation*}
\| \mathcal F \|_{L_2(\mathbb R)} = \sqrt{2 \pi} \| f \|_{L_2(-\mathcal T,\mathcal T)}.
\end{equation*}
and the Riesz-basis property
$$
\left\| \sum f_n e_n^* \right\|_{L_2(-\mathcal T, \mathcal T)} \asymp \| \{ f_n \} \|_{l_2}, \quad \forall \{ f_n \} \in l_2,
$$
concludes the proof.
\end{proof}

Using the relations \eqref{cram} and \eqref{asymptDp}, we get
\begin{equation} \label{findkap}
\kappa_p^N(\nu_n) = \nu_n^{B_p-2} \left( \frac{A_1(\mu_n)}{A(\mu_n)} - \Delta_p^{N,0}(\mu_n)\right),
\quad \kappa_p^D(\nu_n) = \nu_n^{B_p-1} \left( \frac{A_2(\mu_n)}{A(\mu_n)} - \Delta_p^{D,0}(\mu_n)\right).
\end{equation}

Thus, one can construct the characteristic functions $\Delta_p^{N,0}(\la)$ and $\Delta^{D,0}_p(\la)$ for $\sigma \equiv 0$,
find $\kappa_p^N(\rho)$ and $\kappa_p^D(\rho)$ using \eqref{findkap} and \eqref{expF}, and then obtain $\Delta_p^N(\la)$ and $\Delta_p^D(\la)$ by \eqref{asymptDp}. 

\subsection{Reconstruction on the whole tree}

Summarizing the above results, we obtain the following algorithm for solving Inverse Problem~\ref{ip:main} on the whole tree.

\begin{alg} \label{alg:main}
Suppose that the metric tree $G$, the characteristic functions $\Delta(\la)$ and $\Delta_k(\la)$ for all $v_k \in \partial G'$ are given. We have to find the potentials $\{ \sigma_j \}_1^m$.
\begin{enumerate}
\item Initialize the current set of the boundary vertices $V := \partial G'$ and, for each $v_k \in V$, the characteristic functions $\Delta_k^D(\la) := \Delta(\la)$ and $\Delta_k^N(\la) := \Delta_k(\la)$. 
\item For each $v_k \in V$, recover the potential $\sigma_k$ as follows:
\begin{enumerate}
\item Construct the Weyl function $M_k(\la) := -\dfrac{\Delta_k^N(\la)}{\Delta_k^D(\la)}$ and $M_k^0(\la)$ for the tree $G$ with the zero potential $\sigma^0 \equiv 0$.
\item Choose appropriate contours $\Gamma$ and $\ga$ by \eqref{cont} so that all the zeros of $\Delta_k^D(\la)$ and $\Delta_k^{D,0}(\la)$ lie inside $\Gamma$.
\item Construct the functions $r_k^0(x, \rho, \theta)$ and $F_k(x, \rho)$ by \eqref{defrk0} and \eqref{defFk}, respectively.
\item Solving the main equation \eqref{maineq}, find $\psi_k(x, \rho)$.
\item Find $S_k(x, \la)$ by \eqref{intS} and construct $\sigma_k$ by \eqref{recsi0}.
\end{enumerate}
\item Form the new set $V$ of such the vertices $v_p$ that the potentials on their directed subtrees $g_p$ are completely recovered, while the potential $\sigma_p$ is not. If $V = \{ v_{m+1} \}$, then terminate the algorithm.
\item For each $v_p \in V$, choose any $v_k$ such that $e_k = [v_k, v_p]$ and solve Auxiliary Problem~\ref{ap:aux} (meaning that $\Delta = \Delta_k^D$, $\Delta_k = \Delta_k^N$) as follows:
\begin{enumerate}
\item Using the known potentials $\{ \sigma_j \}$ on the subtree $g_p$, find the solutions $S_j(x_j, \la)$ and $\vv_j(x_j, \la)$ ($x_j \in [0,T_j]$) of the initially value problems \eqref{eqv}, \eqref{ic} together with their quasi-derivatives and construct the characteristic functions $\Delta^{DD}(\la)$, $\Delta^{DK}(\la)$, $\Delta^{ND}(\la)$, and $\Delta^{NK}(\la)$ according to Definition~\ref{def:char}.
\item Determine the characteristic functions $\Delta_p^{D,0}(\la)$ and $\Delta_p^{N,0}(\la)$ by Definition~\ref{def:char} for the zero potential $\sigma^0 \equiv 0$.
\item Construct the functions $A(\la)$, $A_1(\la)$, and $A_2(\la)$ by \eqref{defE}, \eqref{defE1}, and \eqref{defE2}, respectively.
\item Find the values $\kappa_p^N(\nu_n)$ and $\kappa_p^D(\nu_n)$ by \eqref{findkap} at the points $\{ \nu_n \}_{n \in \mathbb Z}$ given by \eqref{defnu}. Therein, the values $B_p$ and $\mathcal T$ are the corresponding parameters of the subtree $G_p$ defined by the structure of the metric graph.
\item Recover the functions $\kappa_p^N(\rho)$ and $\kappa_p^D(\rho)$ from their values at the points $\{ \nu_n \}_{n \in \mathbb Z}$ by the formula \eqref{expF}.
\item Determine $\Delta_p^N(\la)$ and $\Delta_p^D(\la)$ by \eqref{asymptDp}.
\end{enumerate}
\item At this step, the functions $\Delta_p^D(\la)$ and $\Delta_p^N(\la)$ have already been found for each $v_p \in V$. Go to step~2.
\end{enumerate}
\end{alg}

In the case of $\Delta(\la)$ and $\Delta_k(\la)$ ($v_k \in \partial G'$) being the spectral data of some potential $\sigma \in L_2(G)$, Algorithm~\ref{alg:main} allows one to uniquely reconstruct this potential according to the arguments of Sections~\ref{sec:bound}--\ref{sec:alg}. Let us show that, moreover, Algorithm~\ref{alg:main} is stable under small perturbations of the spectral data.

\begin{thm} \label{thm:loc}
Let $\sigma = \{ \sigma_j \}_1^m$ be a fixed potential in $L_2(G)$. Then, there exists $\eps > 0$ (depending on $\sigma$) such that, for any functions $\tilde \Delta(\la)$ and $\tilde \Delta_k(\la)$ $(k \colon v_k \in \partial G')$ satisfying the conditions $\rho^{b-1}\hat \Delta(\rho^2) \in PW(\mathbf{T})$, $\rho^{b-2} \hat \Delta_k(\rho^2) \in PW(\mathbf{T})$, and $\de \le \eps$, where $\de$ is defined by \eqref{defde}, Algorithm~\ref{alg:main} is executed correctly and, as a result, uniquely determines some potential $\tilde \sigma \in L_2(G)$. Moreover, the stability estimate \eqref{estsi} holds, where the constant $C$ depends only on $\sigma$.
\end{thm}

\begin{proof}
To prove the correct execution of Algorithm~\ref{alg:main}, we have to show that (i) the main equation at step 2.4 is uniquely solvable and (ii) the sequences $\{ \tilde \kappa_p^N(\nu_n) \}_{n \in \mathbb Z}$ and $\{ \tilde \kappa_p^D(\nu_n) \}_{n \in \mathbb Z}$ at step 4.4 belong to $l_2$.

For $\Delta(\la)$ and $\Delta_k(\la)$ being the characteristic functions of the corresponding eigenvalue problems $\mathcal L$ and $\mathcal L_k$ for the fixed potential $\sigma$, equation \eqref{mainop} is uniquely solvable in view of the explicit construction \eqref{invHk0} of the inverse operator $(I + H_k^0(x))^{-1}$. Suppose that $\tilde \Delta(\la)$ and $\tilde \Delta_k(\la)$ satisfy the hypothesis of this theorem for some $\eps > 0$. If $\eps$ is small enough, then the zeros of $\tilde \Delta(\la)$ lie inside the contour $\Gamma$ chosen for $\Delta(\la)$ at step 2.2. Using \eqref{defde}, \eqref{estDk}, \eqref{defrk0}, and \eqref{defHk0}, we get
\begin{equation} \label{difHk0}
\| \hat{\mathscr M}_k(\theta) \|_{L_2(\ga)} \le C \de, \quad 
\| \hat H_k^0(x) \|_{L_2(\ga) \to L_2(\ga)} \le C \de, \quad x \in [0,T_k].
\end{equation}
It can be shown that $H_k^0(x)$ is continuous by $x \in [0,T_k]$ in the operator norm $\| . \|_{L_2(\ga) \to L_2(\ga)}$. Hence, the operator $(I + H_k^0(x))^{-1}$ is uniformly bounded for $x \in [0,T_k]$. Consequently, for sufficiently small $\eps > 0$ and for each $x \in [0,T_k]$, the operator $(I + \tilde H_k^0(x))$ has a bounded inverse on $L_2(\ga)$. 
In addition, using \eqref{defFk}, we obtain 
\begin{equation} \label{difFk}
\tilde F_k(x,.) \in L_2(\ga) \quad \text{and} \quad \| F(x, .) \|_{L_2(\ga)} \le C \de, \quad x \in [0,T_k]
\end{equation}
Therefore, the equation
$$
(I + \tilde H_k^0(x)) \tilde \psi_k(x) = \tilde F_k(x)
$$
has a unique solution $\tilde \psi_k(x.,) \in L_2(\ga)$ for each $x \in [0,T_k]$, which concludes the proof of (i). Obviously, the estimates \eqref{difHk0} and \eqref{difFk} together with \eqref{eqPW} imply 
\begin{equation} \label{difpsik}
\| \hat \psi_k(x, .) \|_{L_2(\mathbb R)} \le C \de, \quad x \in [0,T_k].
\end{equation}

Next, consider step 4 for some $v_p$ and $v_k$. Assume that 
$$
\de_k + \| \hat \sigma \|_{L_2(g_p)} \le C \de,
$$
which follows from the previous steps of the algorithm. Following the proof of Theorem~\ref{thm:uniaux} in Section~\ref{sec:return}, we obtain the estimates \eqref{estE} and \eqref{estE1}.
\end{proof}

Note that Theorem~\ref{thm:loc} does not assert that $\tilde \Delta(\la)$ and $\tilde \Delta_k(\la)$ are the characteristic functions corresponding to the constructed potential $\tilde \sigma$. Moreover, choosing different $v_k$ at step~4 of Algorithm~\ref{alg:main}, we may get different results $\tilde \sigma$. Anyway, under the assumption of existence for the inverse problem solution, Theorem~\ref{thm:loc} immediately yields Theorem~\ref{thm:locex}.

\section*{Appendix}

\setcounter{equation}{0}
\renewcommand\theequation{A.\arabic{equation}}

\setcounter{thm}{0}
\renewcommand\thethm{A.\arabic{thm}}

Here we provide the proof of Lemma~\ref{lem:uniSC}, based on the technique of \cite{Bond21} for constructing transformation operators.

Consider the Sturm-Liouville equation
\begin{equation} \label{eqv1}
-(y^{[1]})' - \sigma(x) y^{[1]} - \sigma^2(x) y = \lambda y, \quad x \in (0,T),
\end{equation}
where $\sigma$ is a complex-valued function of $L_2(0,T)$, $y^{[1]} := y' - \sigma y$ is the quasi-derivative, and $\la = \rho^2$ is the spectral parameter.

Denote by $\vv(x, \la)$ the solution of equation \eqref{eqv1} satisfying the initial conditions $\vv(0,\la) = 1$, $\vv^{[1]}(0,\la) = 0$.
In \cite{Bond21}, the following representations have been obtained:
\begin{align} \label{transK}
    & \vv(x, \la) = \cos \rho x + \int_0^x \mathscr K(x, t) \cos \rho t \, dt, \\ \label{transN}
    & \vv^{[1]}(x, \la) = -\rho \sin \rho x + \rho \int_0^x \mathscr N(x, t) \sin \rho t \, dt + \mathscr C(x),
\end{align}
where the functions $\mathscr K(x,t)$, $\mathscr N(x,t)$, and $\mathscr C(x)$ are constructed as the series:
\begin{equation} \label{ser}
\mathscr K = \sum_{n = 0}^{\infty} \mathscr K_n, \quad \mathscr N = \sum_{n = 0}^{\infty} \mathscr N_n, \quad \mathscr C = \sum_{n = 0}^{\infty} \mathscr C_n,
\end{equation}
whose terms are defined recursively (see \cite{Bond21, KB25}):
\begin{align} \label{defK0}
    \mathscr K_0(x, t) & = \tfrac{1}{2} \sigma\left( \tfrac{x + t}{2}\right) + \tfrac{1}{2} \sigma\left( \tfrac{x - t}{2}\right) - \tfrac{1}{2} \int_0^x \sigma^2(s) \, ds - \tfrac{1}{4} \int_t^x \left( \sigma^2(\tfrac{x - s}{2}) - \sigma^2(\tfrac{x + s}{2})\right) \, ds, \\ \label{defN0}
    \mathscr N_0(x, t) & = \tfrac{1}{2} \sigma\left( \tfrac{x + t}{2}\right) - \tfrac{1}{2} \sigma\left( \tfrac{x - t}{2}\right) + \tfrac{1}{2} \int_0^x \sigma^2(s) \, ds + \tfrac{1}{4} \int_t^x \left( \sigma^2(\tfrac{x - s}{2}) + \sigma^2(\tfrac{x + s}{2})\right) \, ds, \\ \label{defC0}
    \mathscr C_0(x) & = -\tfrac{1}{2} \int_0^x \sigma^2(t)\,dt - \tfrac{1}{4} \int_0^x (\sigma^2(\tfrac{x + t}{2}) + \sigma^2(\tfrac{x - t}{2})) \, dt, \\ \nonumber
    \mathscr K_{n+1}(x,t) & =  \frac{1}{2} \int_{x - t}^x \bigl( \mathscr K_n(s, t - x + s) + \mathscr N_n(s, t-x+s) \bigr) \sigma(s) \, ds \\ \nonumber & + \frac{1}{2}\int_{\frac{x - t}{2}}^{x - t} \bigl( \mathscr K_n(s, x - s - t) - \mathscr N_n(s, x - s - t) \bigr) \sigma(s) \, ds \\ \nonumber & + \frac{1}{2}\int_{\frac{x + t}{2}}^x \bigl( \mathscr K_n(s, x - s + t) - \mathscr N_n(s, x - s + t) \bigr) \sigma(s) \, ds \\ \nonumber & - \frac{1}{2} \biggl( \int_0^x \sigma^2(s) \, ds \int_0^{\min\{s, x-t\}} \mathscr K_n(s, s-\xi) \, d\xi \\ \nonumber & + \int_0^{x-t} \sigma^2(s) \, ds \int_s^{\min\{ 2s, x-t\}} \mathscr K_n(s, \xi - s) \, d\xi \\ \label{defKn} & - \int_{\frac{x+t}{2}}^x \sigma^2(s) \, ds \int_t^{2s-x} \mathscr K_n(s, x + \xi - s) \, d\xi\biggr)  - \int_0^{x - t} \mathscr C_n(s) \sigma(s) \, ds, \\ \nonumber
    \mathscr N_{n+1}(x,t) & =-\frac{1}{2} \int_{x - t}^x \bigl( \mathscr K_n(s, t - x + s) + \mathscr N_n(s, t - x + s) \bigr) \sigma(s) \, ds \\ \nonumber & - \frac{1}{2}\int_{\frac{x - t}{2}}^{x - t} \bigl( \mathscr K_n(s, x - s - t) - \mathscr N_n(s, x - s - t) \bigr) \sigma(s) \, ds \\ \nonumber & + \frac{1}{2}\int_{\frac{x + t}{2}}^x \bigl( \mathscr K_n(s, x - s + t) - \mathscr N_n(s, x - s + t) \bigr) \sigma(s) \, ds \\ \nonumber & +\frac{1}{2} \biggl( \int_0^x \sigma^2(s) \, ds \int_0^{\min\{s, x-t\}} \mathscr K_n(s, s-\xi) \, d\xi \\ \nonumber & + \int_0^{x-t} \sigma^2(s) \, ds \int_s^{\min\{ 2s, x-t\}} \mathscr K_n(s, \xi - s) \, d\xi \\ \nonumber & + \int_{\frac{x+t}{2}}^x \sigma^2(s) \, ds \int_t^{2s-x} \mathscr K_n(s, x + \xi - s) \, d\xi\biggr) + \int_0^{x - t} \mathscr C_n(s) \sigma(s) \, ds, \\ \nonumber
    \mathscr C_{n+1}(x) & = -\frac{1}{2} \int_0^x \sigma^2(t) \, dt \biggl( \int_{x-t}^x \mathscr K_n(t, \xi - x + t) \, d\xi + \int_{x - 2t}^{x - t} \mathscr K_n(t, x-\xi-t) \, d\xi\biggr) \\ \nonumber &
    - \int_0^x \mathscr C_n(s) \sigma(s) \, ds.
\end{align}

The functions $\mathscr K_n(x,t)$ and $\mathscr N_n(x,t)$ ($n \ge 1$) are continuous on $\mathbf{D} := \bigl\{ (x,t) \colon 0 \le t \le x \le T \bigr\}$ and the functions $\mathscr C_n(x)$ ($n \ge 0$), on $[0,T]$. Moreover, the following estimates hold:
\begin{equation} \label{estKn}
|\mathscr K_n(x, t)|, \, |\mathscr N_n(x, t)|, \, |\mathscr C_n(x)| \le a^n Q^n(x) \sqrt{\tfrac{x^{n-1}}{(n-1)!}}, \quad n \ge 1, \quad (x,t) \in \mathbf{D},
\end{equation}
where $Q(x) = \| \sigma \|_{L_2(0,x)}$ and $a$ is some constant depending on $\| \sigma \|_{L_2(0,T)}$ (see \cite{KB25}). Consequently, the series of $\mathscr K_n$ and $\mathscr N_n$ starting from $n \ge 1$ and the series of $\mathscr C_n$ in \eqref{ser} converge absolutely and uniformly to continuous functions. Adding the corresponding terms $\mathscr K_0$ and $\mathscr N_0$, we conclude that, for each fixed $x \in (0,T]$, the functions $\mathscr K(x,.)$ and $\mathscr N(x,.)$ belong to $L_2(0,x)$.

Now, along with \eqref{eqv}, consider the similar equation with another complex-valued potential $\tilde \sigma \in L_2(0,T)$. Let us show that, for each fixed $x \in (0,T]$, there hold
\begin{equation} \label{difest}
\| \hat{\mathscr K}(x, .) \|_{L_2(0,x)}, \, \| \hat{\mathscr N}(x, .) \|_{L_2(0,x)}, \, |\hat{\mathscr C}(x)| \le C \| \hat \sigma \|_{L_2(0,x)},
\end{equation}
where the constant $C$ depends only on $\max \bigl\{ \| \sigma \|_{L_2(0,T)}, \| \tilde \sigma \|_{L_2(0,T)} \bigr\}$. 

Using \eqref{defK0}--\eqref{defC0}, we get
\begin{equation} \label{est0}
\| \hat{\mathscr K}_0(x,.) \|_{L_2(0,x)}, \, \| \hat{\mathscr N}_0(x,.) \|_{L_2(0,x)}, \,  |\hat{\mathscr C}_0(x)| \le a \| \hat \sigma \|_{L_2(0,x)}. 
\end{equation}

Next, by induction, we obtain the estimates
\begin{equation} \label{estdifK}
|\hat{\mathscr K}_n(x, t)|, \, |\hat{\mathscr N}_n(x, t)|, \, |\hat{\mathscr C}_n(x)| \le a^n \| \hat \sigma \|_{L_2(0,x)} Q^{n-1}(x) \sqrt{\tfrac{x^{n-1}}{(n-1)!}}, \quad n \ge 1, \quad (x,t) \in \mathbf{D},
\end{equation}
where 
\begin{equation} \label{defQ}
Q(x) = \| \sigma_{max} \|_{L_2(0,x)}, \quad \sigma_{max}(x) = \max\{ |\sigma(x)|, |\tilde \sigma(x)| \}
\end{equation}
and $a$ is a constant that depends only on $\max \bigl\{ \| \sigma \|_{L_2(0,T)}, \| \tilde \sigma \|_{L_2(0,T)} \bigr\}$.

For instance, let us prove \eqref{estdifK} for $\hat{\mathscr K}_{n+1}(x,t)$ ($n \ge 0$) assuming that the claimed estimates for $\hat{\mathscr K}_n$, $\hat{\mathscr N}_n$, $\hat{\mathscr C}_n$ are already established. Consider the first term in \eqref{defKn}:
\begin{equation} \label{defIn}
I_n(x,t) := 
\frac{1}{2} \int_{x - t}^x \bigl( \mathscr K_n(s, t - x + s) + \mathscr N_n(s, t-x+s) \bigr) \sigma(s) \, ds.
\end{equation}

Combining \eqref{defK0} and \eqref{defN0}, we get
$$
\mathscr K_0(x,t) + \mathscr N_0(x,t) = \sigma\left( \tfrac{x+t}{2} \right) + \tfrac{1}{2} \int_t^x \sigma^2\left( \tfrac{x+s}{2}\right) \,ds.
$$
Hence
\begin{align*}
I_0(x, t) & = \frac{1}{2} \int_{x-t}^x \sigma\left( \tfrac{t-x+2s}{2}\right) \sigma(s) \, ds + \frac{1}{4} \int_{x-t}^x \sigma(s) \, ds \int_{t-x+s}^s \sigma^2 \left( \tfrac{s+\xi}{2}\right) d\xi, \\
\hat{I}_0(x,t) & = \frac{1}{2} \int_{x-t}^x \hat \sigma\left( \tfrac{t-x+2s}{2}\right) \sigma(s) \, ds + \frac{1}{2} \int_{x-t}^x \tilde \sigma\left( \tfrac{t-x+2s}{2}\right) \hat\sigma(s) \, ds \\ & + \frac{1}{4} \int_{x-t}^x \hat \sigma(s) \, ds \int_{t-x+s}^s \sigma^2 \left( \tfrac{s+\xi}{2}\right) d\xi + \frac{1}{4} \int_{x-t}^x \tilde \sigma(s) \, ds \int_{t-x+s}^s \hat\sigma\left( \tfrac{s+\xi}{2}\right) \Bigl( \sigma\left( \tfrac{s +\xi}{2}\right) + \tilde\sigma\left( \tfrac{s+\xi}{2}\right)\Bigr) d\xi
\end{align*}
Applying the Cauchy-Bunyakovsky-Schwarz inequality, we estimate
\begin{equation} \label{estI1}
|\hat{I}_0(x,t)| \le a \| \hat\sigma \|_{L_2(0,x)}, \quad (x,t) \in \mathbf{D}.
\end{equation}

Analogously, consider the fourth term from \eqref{defKn}:
\begin{align} \nonumber
J_{n}(x, t) & := \frac{1}{2} \int_0^x \sigma^2(s) \, ds \int_0^{\min\{s, x-t\}} \mathscr K_n(s, s-\xi) \, d\xi, \\ \nonumber
\hat{J}_n(x,t) & = \frac{1}{2} \int_0^x \hat\sigma(s) (\sigma(s) + \tilde \sigma(s))\, ds \int_0^{\min\{s, x-t\}} \mathscr K_n(s, s-\xi) \, d\xi \\ \label{difJn} & + \frac{1}{2} \int_0^x \tilde \sigma^2(s) \, ds \int_0^{\min\{s, x-t\}} \hat{\mathscr K}_n(s, s-\xi) \, d\xi.
\end{align}

In view of \eqref{defK0}, the function $\mathscr K_0(x,t)$ is square-integrable for each fixed $x \in (0,T]$ and $\| \mathscr K_0(x, .) \|_{L_2(0,x)} \le a$. Using the latter estimate and \eqref{est0}, we obtain
\begin{equation} \label{estintK0}
\left| \int_0^{\min\{s, x-t\}} \mathscr K_n(s, s-\xi) \, d\xi \right| \le a, \quad \left| \int_0^{\min\{s, x-t\}} \hat{\mathscr K}_n(s, s-\xi) \, d\xi \right| \le a \| \hat\sigma \|_{L_2(0,x)}.
\end{equation}

Using \eqref{difJn} for $n = 0$ and \eqref{estintK0}, we deduce
\begin{equation} \label{estJ1} 
|\hat{J}_0(x,t)| \le a \| \hat\sigma \|_{L_2(0,x)}, \quad (x,t) \in \mathbf{D}.
\end{equation}

Estimates similar to \eqref{estI1} and \eqref{estJ1} can be obtained for the other terms in \eqref{defKn}, so we arrive at \eqref{estdifK} for $\hat{\mathscr K}_1(x,t)$.

Proceed to proving \eqref{estdifK} for $\mathscr K_{n+1}(x,t)$ with $n \ge 1$. As above, we confine ourselves by analyzing the terms $I_n$ and $J_n$, since the other terms in \eqref{defKn} can be treated analogously. It follows from \eqref{defIn} that
\begin{align} \nonumber
\hat{I}_n(x,t) & = \frac{1}{2} \int_{x-t}^x \bigl(\hat{\mathscr K}_n(s,t-x+s) + \hat{\mathscr N}_n(s,t-x+s)\bigr) \sigma(s) \, ds \\ \label{difIn}
& + \frac{1}{2} \int_{x-t}^x \bigl(\tilde{\mathscr K}_n(s,t-x+s) + \tilde{\mathscr N}_n(s,t-x+s)\bigr) \hat \sigma(s) \, ds.
\end{align}

Note that, for $\tilde{\mathscr K}_n(x,t)$ and $\tilde{\mathscr N}_n(x,t)$, the estimates \eqref{estKn} with $Q(x)$ given by \eqref{defQ} are valid. So, using \eqref{estKn}, \eqref{estdifK}, and \eqref{difIn}, we obtain
\begin{align} \label{smIn}
|\hat{I}_n(x,t)| & \le a^n \| \hat \sigma \|_{L_2(0,x)} \int_{x-t}^x Q^{n-1}(s) \sqrt{\frac{s^{n-1}}{(n-1)!}} |\sigma(s)| \, ds + a^n \int_{x-t}^x Q^n(s) \sqrt{\frac{s^{n-1}}{(n-1)!}} |\hat \sigma(s)| \, ds.
\end{align}
Taking \eqref{defQ} into account and applying the Cauchy-Bunyakovsky-Schwarz inequality, we get
\begin{align} \nonumber
& \int_0^x Q^{n-1}(s) \sqrt{\frac{s^{n-1}}{(n-1)!}} |\sigma(s)| \, ds \\ \label{sm1} & \le \sqrt{\int_0^x \sigma_{max}^2(s) \, ds \int_0^s \sigma_{\max}^{2(n-1)}(\xi) \, d\xi} \cdot \sqrt{\int_0^x \frac{s^{n-1}}{(n-1)!}\, ds} = \frac{Q^n(x)}{\sqrt{n}} \cdot \sqrt{\frac{x^n}{n!}}, \\ \nonumber
& \int_0^x Q^n(s) \sqrt{\frac{s^{n-1}}{(n-1)!}} |\hat \sigma(s)| \, ds \\ \label{sm2}
& \le Q^n(s) \sqrt{\int_0^x |\hat \sigma(s)|^2 \, ds} \cdot \sqrt{\int_0^x \frac{s^{n-1}}{(n-1)!}\, ds} = \| \hat \sigma\|_{L_2(0,x)} Q^n(x) \sqrt{\frac{x^n}{n!}}.
\end{align}
Combining \eqref{smIn}, \eqref{sm1}, and \eqref{sm2}, we obtain
\begin{equation} \label{estIn}
|\hat{I}_n(x,t)| \le a^{n+1} \| \hat \sigma \|_{L_2(0,x)} Q^n(x) \sqrt{\frac{x^n}{n!}}, \quad (x,t) \in \mathbf{D}.
\end{equation}

Proceed to estimating $\hat{J}_n$. It follows from \eqref{estKn}, \eqref{estdifK}, and \eqref{difJn} that
\begin{align*}
|\hat{J}_n(x,t)| & \le a^n \int_0^x |\hat \sigma(s)| \sigma_{max}(s) Q^n(s) \sqrt{\frac{s^{n-1}}{(n-1)!}} s \, ds \\ & + 
\frac{a^n}{2} \int_0^x |\tilde \sigma(s)|^2 \| \hat\sigma \|_{L_2(0,s)} Q^{n-1}(s) \sqrt{\frac{s^{n-1}}{(n-1)!}} s \, ds \\ & \le
a^n \sqrt{\frac{x^{n+1}}{(n-1)!}} \Biggl( \int_0^x |\hat \sigma(s)| \cdot \sigma_{max}(s) Q^n(s) \, ds \\ & + \frac{1}{2} \| \hat \sigma \|_{L_2(0,x)} \int_0^x |\tilde \sigma(s)| \cdot \sigma_{max}(s) Q^{n-1}(s) \, ds\Biggr).
\end{align*}
Applying the Cauchy-Bunyakovsky-Schwarz inequality to the latter integrals similarly to \eqref{sm1}, we arrive at the estimate for $\hat{J}_n$ analogous to \eqref{estIn}. Consequently, we obtain \eqref{estdifK} for $\hat{\mathscr K}_{n+1}(x,t)$.

Summing up the series
$$
\hat{\mathscr K} = \sum_{n = 0}^{\infty} \hat{\mathscr K}_n, \quad
\hat{\mathscr N} = \sum_{n = 0}^{\infty} \hat{\mathscr N}_n, \quad
\hat{\mathscr C} = \sum_{n = 0}^{\infty} \hat{\mathscr C}_n
$$
and using \eqref{est0}--\eqref{estdifK}, we prove \eqref{difest}. Clearly, the estimates \eqref{difest} imply \eqref{univv} in Lemma~\ref{lem:uniSC}. The estimates \eqref{uniS} are obtained analogously.

\medskip

{\bf Funding.} This work was supported by Grant 24-71-10003 of the Russian Science Foundation, https://rscf.ru/en/project/24-71-10003/.

\medskip

{\bf Acknowledgement.} The author is grateful to Professor Maria A. Kuznetsova for her careful reading and valuable comments.

\noindent Natalia Pavlovna Bondarenko \\

\noindent 1. Department of Mechanics and Mathematics, Saratov State University, 
Astrakhanskaya 83, Saratov 410012, Russia, \\

\noindent 2. Department of Applied Mathematics, Samara National Research University, \\
Moskovskoye Shosse 34, Samara 443086, Russia, \\

\noindent 3. S.M. Nikolskii Mathematical Institute, RUDN University, 6 Miklukho-Maklaya St, Moscow, 117198, Russia, \\

\noindent 4. Moscow Center of Fundamental and Applied Mathematics, Lomonosov Moscow State University, Moscow 119991, Russia.\\

\noindent e-mail: {\it bondarenkonp@sgu.ru}

\end{document}